\pdfoutput=1
\documentclass[a4paper,reqno, oneside]{amsart}
\usepackage[utf8]{inputenc}
\usepackage[T1]{fontenc}
\numberwithin{equation}{section}

\usepackage{mathtools, lmodern, mathabx, shuffle, microtype, bookmark, cleveref} %
\usepackage{mathrsfs}
\usepackage[left=2cm,right=2cm]{geometry}
\usepackage{mleftright}
\mleftright
\usepackage[ocgcolorlinks]{ocgx2}
\hypersetup{citecolor=blue}
\usepackage{placeins}

\usepackage[usenames,dvipsnames,table]{xcolor}
\usepackage{url}
\usepackage{tikz-cd}

\usepackage{enumitem}
\setlist[enumerate]{label=\roman*.}

\usepackage[refpage,notocbasic]{nomencl}

\usepackage[mathlines]{lineno}
\usepackage{etoolbox}
\newcommand*\linenomathpatchAMS[1]{%
  \expandafter\pretocmd\csname #1\endcsname {\linenomathAMS}{}{}%
  \expandafter\pretocmd\csname #1*\endcsname{\linenomathAMS}{}{}%
  \expandafter\apptocmd\csname end#1\endcsname {\endlinenomath}{}{}%
  \expandafter\apptocmd\csname end#1*\endcsname{\endlinenomath}{}{}%
}
\expandafter\ifx\linenomath\linenomathWithnumbers
  \let\linenomathAMS\linenomathWithnumbers
  \patchcmd\linenomathAMS{\advance\postdisplaypenalty\linenopenalty}{}{}{}
\else
  \let\linenomathAMS\linenomathNonumbers
\fi
\linenomathpatchAMS{align}

\def\dparenl{\mathopen{(\mkern-3mu(}}
\def\dparenr{\mathclose{)\mkern-3mu)}}
\newcommand\lie[2]{\mathfrak{g}_{\le #1}\dparenl\R^{#2}\dparenr}

\DeclareMathOperator{\IIS}{IIS}

\theoremstyle{plain}
\newtheorem{theorem}{Theorem}[section]
\newtheorem{conjecture}[theorem]{Conjecture}
\newtheorem{corollary}[theorem]{Corollary}

\newtheorem{lemma}[theorem]{Lemma}
\newtheorem{proposition}[theorem]{Proposition}

\theoremstyle{definition}
\newtheorem{definition}[theorem]{Definition}
\newtheorem{example}[theorem]{Example}

\theoremstyle{remark}
\newtheorem{remark}[theorem]{Remark}

\newcommand{\R}{\mathbb{R}}
\newcommand{\C}{\mathbb{C}}

\newcommand{\Q}{\mathbb{Q}}

\newcommand{\GL}{\operatorname{GL}}
\newcommand{\SO}{\operatorname{SO}}
\renewcommand{\O}{\operatorname{O}}
\newcommand\proj{\operatorname{proj}}

\newcommand{\id}{\operatorname{id}}
\newcommand\SL{\operatorname{SL}}

\def\word#1{\ensuremath{{\color{cyan}\mathtt{#1}}}}
\newcommand\w[1]{\word{#1}}

\newcommand{\K}{\mathcal{K}}
\newcommand{\Skew}[2]{\mathfrak{so}_{#1}(#2)}
\newcommand{\CdSkew}{\C^d\oplus \Skew{d}{\C}}
\newcommand{\RdSkew}{\R^d\oplus \Skew{d}{\R}}

\newcommand{\SOdR}{\SO_d(\R)}

\newcommand{\OiC}{\O^i_d(\C)}
\newcommand{\OdC}{\O_d(\C)}
\newcommand{\OdR}{\O_d(\R)}
\DeclareMathOperator{\ISS}{ISS}

\newcommand{\cc}{\mathbf{c}}
\newcommand{\cctwo}{\mathbf{c}_{\leq 2}}
\newcommand{\ccn}{\mathbf{c}_{\leq n}}

\newcommand{\rhotwo}{\tilde\rho_2}
\newcommand{\rhotwon}{\rho_2}
\newcommand{\rhothree}{\tilde\rho_3}
\newcommand{\rhothreen}{\rho_3}
\newcommand{\rhodn}{\rho_d}
\newcommand{\rhod}{\tilde\rho_d}

\newcommand{\pol}{q}
\newcommand{\polmap}{r}

\newcommand\DOT[2]{#1\cdot #2}

\newcommand\G{{\mathscr G}}
\newcommand\g{{\mathfrak g}}
\newcommand{\lyndon}[1]{\mathscr{L}_{#1}}

\newcommand{\vv}{c}
\newcommand{\m}{c}

\usepackage[linecolor=white,backgroundcolor=white,bordercolor=white,textsize=tiny]{todonotes}
\let\todon\todo
\renewcommand{\todo}[1]{\todon{\color{red}#1}}

\makenomenclature
\begin{document}

\title[The moving frame method for iterated integrals]{The moving frame method for iterated integrals: orthogonal invariants}
\author{Joscha Diehl}
\address{Joscha Diehl\newline Universit\"at Greifswald, Institut f\"ur Mathematik und Informatik, Walther-Rathenau-Str.~47, 17489 Greifswald, Germany.}
\email{joscha.diehl@uni-greifswald.de}
\urladdr{https://diehlj.github.io}
\author{Rosa Preiß}
\address{Rosa Preiß\newline Institut für Mathematik, Technische Universität Berlin}
\email{preiss@math.tu-berlin.de}
\urladdr{http://page.math.tu-berlin.de/~preiss/}
\author{Michael Ruddy}
\address{Michael Ruddy\newline The Data Institute, University of San Francisco}
\email{mruddy@usfca.edu}
\urladdr{https://mgruddy.wixsite.com/home}
\author{Nikolas Tapia}
\address{Nikolas Tapia\newline Weierstra{\ss}-Institut Berlin, Mohrenstr.~39, 10117 Berlin, Germany\newline Technische Universität Berlin, Str. des 17.~Juni 136, 10623 Berlin, Germany.}
\email{tapia@wias-berlin.de}
\urladdr{https://www.wias-berlin.de/people/tapia}
\thanks{R.P. is supported by the European Research Council through CoG-683164}
\thanks{M.R. was supported in part by the Max Planck Institute for Mathematics in the Sciences}
\thanks{N.T. is supported by the DFG MATH$^+$ Excellence Cluster}
\begin{abstract}
    Geometric features robust to noise of curves in Euclidean space are of great interest for various applications such as machine learning and image analysis.
    We apply Fels-Olver's moving frame method (for geometric features) paired with the log-signature transform
    (for robust features) to construct a set of integral invariants under rigid motions for curves in $\R^d$ from the iterated-integral signature. In particular we show that one can algorithmically construct a set of invariants that characterize the equivalence class of the truncated iterated-integrals signature under orthogonal transformations which yields a characterization of a curve in $\R^d$ under rigid motions (and tree-like extensions) and an explicit method to compare curves up to these transformations.
\end{abstract}
\subjclass[2020]{60L10\and 14L24}
\keywords{Signature\and geometric invariants\and moving frame\and orthogonal group.}

\newcommand\e{\mathsf{e}} %

\newcommand\INV{{\mathfrak I}} %
\newcommand{\IM}{\INV_M}

\maketitle

\tableofcontents

\section{Introduction}

A central problem in image science is constructing geometrically relevant features of curves that are robust to noise. In this sense, rigid motions of space make up a natural group of `nuisance' transformations of the data. For this reason, rotation- and translation-invariant features are often desired, for instance, in Human Activity Recognition \cite[Section 6]{morales2017physical}
or in matching contours \cite{zhang2018multi}. Classically, differential invariants such as curvature have been used \cite{HO13} for this purpose, and more recently integral invariants of curves have been of interest \cite{FKK08, DR2019}. In this work we construct a rigid motion-invariant representation of a curve through its \emph{iterated-integrals signature} by applying the \emph{Fels-Olver moving frame method}. We show that this yields sets of integral invariants that characterize the truncated iterated integral signature up to rigid motions.

In \cite{Che54, Che58}, the author used the collection of all iterated integrals to characterize smooth curves, and in \cite{hambly2010uniqueness} the author extended this result to more irregular curves. The modern term for this collection of iterated integrals of a curve is the iterated-integrals signature. It has since been used in various applications such as constructing features for machine learning tasks (see \cite{CK2016} and references therein) and shape analysis \cite{CLT2019, LG2020}.

The Fels-Olver moving frame method, introduced in \cite{FO99}, is a modern generalization of the classical moving frame method formulated by Cartan \cite{C35}. In the general setting of a Lie group $G$ acting on a manifold $M$ a moving frame is defined as a $G$-equivariant map from $M$ to $G$. A moving frame can be re-interpreted as a choice of cross-sections to the orbits of $G$, and hence a unique canonical form for elements of $M$ under $G$. Thus the moving frame method provides a framework for algorithmically constructing $G$-invariants on $M$ that characterize orbits and for determining equivalence of submanifolds of $M$ under $G$.

The moving frame method has been used to construct differential invariants of smooth planar and spatial curves under Euclidean, affine, and projective transformations, and, in certain cases, these differential invariants lead to a \emph{differential signature} which can be used to classify curves under these transformation groups \cite{CO98}. The differential signature has been applied in a variety of image science applications from automatic jigsaw puzzle assembly \cite{HO14} to medical imaging \cite{GS17}. Also in the realm of image science, the moving frame method has been used to construct invariants of grayscale images \cite{BH13, TOT19}.

We consider the induced action of the orthogonal group of rotations on the log-signature of a curve, which provides a compressed representation of a curve obtained by applying the log transform to the iterated-integrals signature, and provide an explicit cross-section for this action. We show that for most curves and any truncation of the curve's log-signature, the orbit is characterized by the value on this cross-section. As a consequence a curve is completely determined up to rigid motions and tree-like extensions by the invariantization of its iterated-integrals signature induced by this cross-section.

This yields a constructive method to compare curves up to rigid motions and to evaluate rotation invariants that characterize the iterated-integrals signature under rotations.
These invariants are constructed from integrals on the curve, and hence are likely to be more noise-resistant than their differential counterparts such as curvature.
One can easily set up an artificial example where this is visible.
Consider for instance the circle of radius $n^{-3/2}$ given by the parameterization $\gamma: [0,1] \rightarrow \R^2$ where 
$$
\gamma(t) = (x(t), y(t)) = \left(\frac{\cos(2\pi n t)}{n^{3/2}},\frac{\sin(2\pi n t)}{n^{3/2}}  \right),
$$
which as $n\rightarrow \infty$, converges to the constant curve (at the origin).
Now the curvature of this curve does \emph{not} converge (in fact, it blows up).
In contrast, the iterated integrals do all converge (to zero) since $\gamma$ converges in variation norm.
Then, the invariants built out of the iterated integrals (\Cref{sec:MFsection}) also
converge to their value on the zero-curve. On this toy example these integral invariants are hence more ``stable''.

Additionally, in contrast to the methods in \cite{DR2019},
the resulting set of integral invariants is shown to uniquely characterize the curve under rotations, and moreover, does so in a minimal fashion. Since the iterated-integrals signature of a curve is automatically invariant to translations,
this provides rigid motion-invariant features of a curve which can be used for applications such as machine learning or shape analysis.

This work is structured as follows. In Section \ref{sec:prelim} we provide background on the iterated-integrals signature and the moving frame method, as well as some facts about algebraic group and invariants.
In Section \ref{sec:complexdetour} and consider the orthogonal action on the second order truncation of the log-signature over the complex numbers. Using tools from algebraic invariant theory, we construct the linear space which will form the basis for the cross-section in the following section. We also provide an explicit set of polynomial invariants that characterize the second order truncation of the log-signature under the orthogonal group. In Section \ref{sec:MFsection} we construct the moving frame map for paths in $\R^d$. In particular, in Section \ref{sec:MF2}, we outline our procedure and results in simple language for paths in $\R^2$ and in Section \ref{sec:towardpoly} we introduce sufficient conditions for the resulting moving frame invariants to be \emph{polynomial}, showing these conditions are satisfied for some values of $d$.
In Section \ref{sec:inv.planar} we detail the moving frame map for planar and space curves, compute some of the resulting integral invariant functions, and illustrate this procedure on a particular curve. Finally in Section \ref{sec:discussion} we discuss some of the interesting questions that arise as a result of our work.

\section{Preliminaries}\label{sec:prelim}

\subsection{The tensor algebra}

Let $d \ge 1$ be an integer.
A \textbf{word}, or {multi-index}, over the alphabet $\{\w{1},\dotsc,\w{d}\}$
is a tuple $w=(\w{w}_1,\dotsc,\w{w}_n)\in\{\w{1},\dotsc,\w d\}^n$ for some integer $n\ge 0$, called its \textbf{length} which is denoted by $|w|$.
As is usual in the literature, we use the short-hand notation $w=\w w_1\dotsm\w w_n$, where the $\w w_i$, words of length one, are called \textbf{letters}.
The \textbf{concatenation} of two words $v,w$ is the word $vw\coloneq \w v_1\dotsm\w v_n\w
w_1\dotsm\w w_m$ of length $|vw|=n+m$.
Observe that this product is associative and non-commutative.
There is a unique element of length zero, called the empty word and denoted by $\e$.
It satisfies $w\e=\e w=w$ for all words $w$.
If we denote by $T(\R^d)$ the real vector space spanned by words, the bilinear extension of the
concatenation product endows it with the structure of an associative (and non-commutative) algebra.
We also note that \(T(\R^d)\) admits the direct sum decomposition
\[
  T(\R^d)=\bigoplus_{k=0}^\infty\operatorname{span}_{\mathbb R}\{w:|w|=k\}.
\]
There is a commutative product on \(T(\R^d)\), known as the \textbf{shuffle product}, recursively
defined by \(\e\shuffle w\coloneq w\eqcolon w\shuffle \e\) and
\[
    v\w i\shuffle w\w j\coloneq(v\shuffle w\w j)\w i+(v\w i\shuffle w)\w j.
\]

The commutator \textbf{bracket} \([u,v]\coloneq uv-vu\) endows \(T(\R^d)\) with the structure of a Lie
algebra.
The \textbf{free Lie algebra over $\R^d$}, denoted by $\g(\R^d)$, can be realized as the following subspace of
$T(\R^d)$,
\begin{align*}
    \g(\R^d)=\bigoplus_{n=1}^\infty W_n
\end{align*}
where \(W_1\coloneq\operatorname{span}_\R \{\w1,\dotsc,\w d\} \cong \R^d\) and \(W_{n+1}\coloneq[W_1,W_n]\).
There are multiple choices of bases for \(\g(\R^d)\), but we choose to work with the
\textbf{Lyndon basis}. A \textbf{Lyndon word} is a word \(h\) such that whenever \(h=uv\), with
\(u,v\neq\e\), then \(u<v\) for the lexicographical order.
We denote the set of Lyndon words over the alphabet \(\{\w1,\dotsc,\w d\}\) by \(\lyndon{d}\)\nomenclature[Ld]{$\lyndon{d}$}{The Lyndon words over the alphabet $\{\w1,\ldots,\w d\}$}.
In particular, \(h\) with \(|h|\ge 2\) is Lyndon if and only if there exist non-empty Lyndon words \(u\) and \(v\)
such that \(u<v\) and \(h=uv\). Although there might be multiple choices for this factorization, the one with
\(v\) as long as possible is called the \textbf{standard factorization} of \(h\).
The Lyndon basis \(b_{\w{h}}\) is recursively defined by setting \(b_{\w{i}}=\w i\) and
\(b_{h}=[b_u,b_v]\) for all Lyndon words \(h\) with \(|h|\ge 2\), where \(h=uv\) is the standard
factorization.
\begin{example}
	Suppose \(d=2\). The Lyndon words up to length 4, their standard factorizations and the associated
	basis elements are
	\begin{table}[!ht]
		\centering
		\begin{tabular}{|c|c|c|c|}
			\hline
			\(h\)&\(u\)&\(v\)&\(b_h\)\\
			\hline
			\w{1}&---&---&\w{1}\\
			\w{2}&---&---&\w{2}\\
			\w{12}&\w{1}&\w{2}&\([\w{1},\w{2}]\)\\
			\w{112}&\w{1}&\w{12}&\([\w{1},[\w{1},\w{2}]]\)\\
			\w{122}&\w{12}&\w{2}&\([[\w1,\w2],\w2]\)\\
			\w{1112}&\w{1}&\w{112}&\([\w1,[\w1,[\w1,\w2]]]\)\\
			\w{1122}&\w{1}&\w{122}&\([\w1,[[\w1,\w2],\w2]]\)\\
			\w{1222}&\w{122}&\w{2}&\([[[\w1,\w2],\w2],\w2]\)\\
			\hline
		\end{tabular}
	\end{table}
	\label{ex:lyndon}
\end{example}

Elements of the dual space \(T\dparenl\R^d\dparenr\coloneq T(\R^d)^*\) can be identified with formal word
series. For \(F\in T\dparenl\R^d\dparenr\) we write
\[
    F=\sum_w\langle F,w\rangle w.
\]
In particular, we have no growth requirement for the \textbf{coefficients} \(\langle F,w\rangle \in \R\).
The above expression is meant only as a notation for treating the values of \(F\) on words as a single object.
This space can be endowed with a multiplication given, for \(F,G\in T\dparenl\R^d\dparenr\), by
\begin{equation}
	FG=\sum_{w}\left( \sum_{uv=w}\langle F,u\rangle\langle G,v\rangle \right)w.
	\label{eq:convprod}
\end{equation}
Observe that since there is a finite number of pairs of words \(u,v\) such that \(uv=w\), the
coefficients of \(FG\) are well defined for all \(w\), so the above formula is an honest element of
\(T\dparenl\R^d\dparenr\).
It turns out that this product is dual to the \textbf{deconcatenation coproduct} \(\Delta\colon
T(\R^d)\to T(\R^d)\otimes T(\R^d)\) given by
\[
	\Delta w=\sum_{uv=w}u\otimes v,
\]
in the sense that
\[
	\langle FG,w\rangle=\langle F\otimes G,\Delta w\rangle
\]
for all words. This formula is nothing but \cref{eq:convprod} componentwise.

There are two distinct subsets of \(T\dparenl\R^d\dparenr\) that will be important in what follows.
The first one is the subspace \(\g\dparenl\R^d\dparenr\) of \textbf{infinitesimal
characters}, formed by linear maps \(F\) such that \(\langle F,u\shuffle v\rangle=0\) whenever \(u\)
and \(v\) are non-empty words, and such that \(\langle F,\e\rangle=0\).
It can be identified with the dual space
\[
    \g\dparenl\R^d\dparenr\cong\g(\R^d)^*\cong\prod_{n=1}^\infty W_n.
\]
It is a Lie algebra under the commutator bracket \([F,G]=FG-GF\).
The second one is the set \(\G\dparenl\R^d\dparenr\) of \textbf{characters}, i.e., linear maps
\(F\) such that \(\langle F,u\shuffle v\rangle=\langle F,u\rangle\langle F,v\rangle\) for all
\(u,v\in T(\R^d)\).

We may define an exponential map \(\exp\colon\g\dparenl\R^d\dparenr\to\G\dparenl\R^d\dparenr\) by its power series
\[
    \exp(F)\coloneq\sum_{n=0}^\infty\frac{1}{n!}F^n.
\]
On a single word, the map is given by
\[
    \langle\exp(F),w\rangle=\sum_{n=0}^\infty\frac{1}{n!}\left( \sum_{v_1\dotsm v_n=w}\langle
    F,v_1\rangle\dotsm\langle F,v_n\rangle \right),
\]
and since \(F\) vanishes on the empty word, all terms with \(n>|w|\) also vanish, so that the sum is
always finite.
Therefore, \(\exp(F)\) is a well defined element of \(T\dparenl\R^d\dparenr\).
It can be shown that the image of \(\exp\) is equal to $\G\dparenl\R^d\dparenr$ and that it is a bijection onto its image,
with inverse \(\log\colon\G\dparenl\R^d\dparenr\to\g\dparenl\R^d\dparenr\) defined by
\[
    \log(G)\coloneq\sum_{n=1}^\infty\frac{(-1)^{n-1}}{n}(G-\varepsilon)^n
\]
where \(\varepsilon\) is the unique linear map such that \(\langle\varepsilon,\e\rangle=1\) and zero
otherwise.

Finally, we remark some freeness properties of the tensor algebra and its subspaces.
Below,
\[
	T^+(\R^d)=\bigoplus_{n>0}(\R^d)^{\otimes n}
\]
denotes the reduced tensor algebra over \(\R^d\).
The following result can be found in \cite[Corollary 2.1]{FPT2016}.
\begin{proposition}
	Let \(\phi\colon T^+(\R^d)\to\R^e\) be a linear map.
	There exists a unique extension \(\tilde\phi\colon T(\R^d)\to T(\R^e)\) such that
	\[
		(\tilde\phi\otimes\tilde\phi)\circ\Delta=\Delta\circ\tilde\phi
	\]
	and \(\pi\circ\tilde{\phi}=\phi\), where \(\pi\colon T(\R^e)\to\R^e\) denotes the projection of
	\(T(\R^e)\) onto \(\R^e\), orthogonal to \(\R\e\) and \(\bigoplus_{n>2}\operatorname{span}_\R\{w:|w|=n\}\).
	Moreover, it is given by
	\[
		\tilde\phi(w)=\sum_{k=1}^{|w|}\sum_{v_1\dotsm v_k=w}\phi(v_1)\dotsm\phi(v_k).
	\]

	By transposition, we obtain a unique map \(\Phi\colon T\dparenl\R^e\dparenr\to T\dparenl\R^d\dparenr\) such that
	\[
		\Phi(FG)=\Phi(F)\Phi(G)
	\]
	for all \(F,G\in T\dparenl\R^e\dparenr\).
	In particular,
	\begin{equation}
		\Phi(F)=\sum_w\left( \sum_{k=1}^{|w|}\sum_{v_1\dotsm v_k=w}\langle
		F,\phi(v_1)\dotsm\phi(v_k)\rangle\right)w.
	\end{equation}
	\label{prp:fext}
\end{proposition}

\subsection{The iterated-integrals signature}
\label{ss:iis}

The iterated-integrals signature of (smooth enough) paths was
introduced by Chen for homological considerations on loop space, \cite{Chen1954}.
It played a vital role in the rough path analysis of Lyons,
a pathwise approach to stochastic analysis, \cite{Lyons1998differential}.
Recently it has found applications in statistics and machine learning,
where it serves as a method of feature extraction for possibly non-smooth time-dependent data.

\newcommand\CURVE{Z}
Let $\CURVE = (\CURVE^\w{1}, \dots, \CURVE^\w{d}) \colon[0,1]\to\R^d$ be an absolutely continuous path.%
\footnote{One can get away with much less regularity, see \cite{Lyons1998differential}. Since our considerations are purely algebraic, there is no loss in restricting to 'smooth' paths.}
Given a word $w=\w{w}_1\dotsm\w{w}_n$, define
\begin{align}
  \label{eq:IIS}
  \langle\IIS(\CURVE), w\rangle \coloneq\idotsint\limits_{0<s_1<\dotsb<s_n<1}\dot \CURVE^{\w{w}_1}(s_1)\dotsm \dot{\CURVE}^{\w{w}_n}(s_n)\,\mathrm ds_1\dotsm\mathrm ds_n\in\R.
\end{align}
This definition has a unique linear extension to $T(\R^d)$.
We obtain thus an element $\IIS(\CURVE) \in T\dparenl\R^d\dparenr$\nomenclature[IIS]{$\IIS(\CURVE)$}{The iterated-integrals signature of the curve $\CURVE$}, called the \textbf{iterated-integrals signature (IIS)} of $\CURVE$.

It was shown by Ree \cite{Ree1958} that the coefficients of $\IIS(\CURVE)$ satisfy the so-called \textbf{shuffle relations}:
\begin{align*}
  \langle \IIS(\CURVE), v \rangle\langle \IIS(\CURVE), w \rangle =\langle \IIS(\CURVE), {v\shuffle w}\rangle.
\end{align*}
In other words, \(\IIS(\CURVE)\in\G\dparenl\R^d\dparenr\).

As a consequence of the shuffle relation one obtains that the \textbf{log-signature}
\(\log(\IIS(\CURVE))\)\nomenclature[log(IIS)]{$\log(\IIS(\CURVE))$}{The log-signature of the curve $\CURVE$} is a \textbf{Lie series}, i.e., an element of \(\g\dparenl\R^d\dparenr\).
Moreover, the identity $\IIS(\CURVE) = \exp\left( \log( \IIS(\CURVE) ) \right)$ holds.
The log-signature therefore contains the same amount of information as the signature itself;
it in fact is a minimal (linear) depiction of it.\footnote{Minimality follows from Chow's theorem, \cite[Theorem 7.28]{friz2010multidimensional}.}

The entire iterated-integrals signature $\IIS(\CURVE)$ is an infinite dimensional object,
and hence can never actually be numerically computed.
We now provide more detail on the truncated, finite-dimensional setting.

For each integer \(N\ge 1\), the subspace \(I_n\subset T\dparenl\R^d\dparenr\) generated by formal
series such that \(\langle F,w\rangle=0\) for all words with \(|w|\le N\) is a two-sided ideal, that
is, the inclusion
\[
	I_n T\dparenl\R^d\dparenr+T\dparenl\R^d\dparenr I_n\subset I_n
\]
holds.
Therefore, the quotient space \(T_{\le n}\dparenl\R^d\dparenr\coloneq T\dparenl\R^d\dparenr/I_n\) inherits an algebra structure from
\(T\dparenl\R^d\dparenr\).
Moreover, it can be identified with the direct sum
\[
	T_{\le n}\dparenl\R^d\dparenr\cong\bigoplus_{k=0}^N\operatorname{span}_{\mathbb R}\{ w:|w|=k\}.
\]
We denote by \(\proj_{\le n}\colon T\dparenl\R^d\dparenr\to T_{\le n}\dparenl\R^d\dparenr\)\nomenclature[projn]{$\proj_{\le n}$}{The canonical projection $\proj_{\le n}\colon T\dparenl\R^d\dparenr\to T_{\le n}\dparenl\R^d\dparenr$} the
canonical projection.

Denote with \(\lie{n}d\)\nomenclature[g<n]{$\lie{n}d$}{The free step-$n$ nilpotent Lie algebra over $\R^d$} the \textbf{free step-\(N\) nilpotent Lie algebra} (over $\R^d$).
It can be realized as the following subspace of $T_{\le n}\dparenl\R^d\dparenr$, see \cite[Section 7.3]{friz2010multidimensional},
\begin{align*}
	\lie{n}d=\bigoplus_{k=1}^N W_k,
\end{align*}
where, as before  \(W_1\coloneq\operatorname{span}_\R \{ \w{i} : i =1 ,\dots, d\} \cong \R^d\) and \(W_{n+1}\coloneq[W_1,W_n]\).
In the case of $N=2$ this reduces to
\begin{align*}
  W_1 \oplus W_2
  \cong
  \RdSkew,
\end{align*}
where we denote with $\mathfrak{so}(d,\R)$\nomenclature[so(d,R)]{$\mathfrak{so}(d,\R)$}{The space of skew-symmetric $\R^{d\times d}$ matrices} the space of skew-symmetric $d\times d$ matrices.
Indeed,
an isomorphism is given by
\begin{equation}\label{eq:lie2isom}
  \sum_{1\le i \le d} c_{i} \w{i}
  +
  \sum_{1\le i<j \le d} c_{ij} \left[ \w{i}, \w{j} \right]
  \,\,\mapsto\,\,
  \left(
  \begin{bmatrix}
    c_{1}\\
    \vdots\\
    c_{d}\\
  \end{bmatrix},
  \begin{bmatrix}
    0&  c_{12} & \cdots &c_{1d}\\    
    - c_{12}& 0 & \cdots &c_{2d}\\    
    \cdots& \cdots & \ddots &\cdots\\    
    -  c_{1d}& -c_{2d} & \cdots &0
  \end{bmatrix}\right).
\end{equation}
We remark that the coefficients \(c_{i}\) and \(c_{ij}\) are the coordinates\footnote{These are often referred to as coordinates of the first kind, see \cite{kawski2011chronological,owren2001integration}} with respect to
the Lyndon basis (see \Cref{ex:lyndon}).

The linear space $\lie{n}d$ is in bijection to its image under the exponential map.
This image, denoted $\G_{\le n}(\R^d)\coloneq\exp\lie{n}d$, is the \textbf{free step-\(N\) nilpotent group} (over $\R^d$).
It is exactly the set of all points in $T_{\le n}(\R^d)$ that can be reached
by the truncated signature map, that is (see \cite[Theorem 7.28]{friz2010multidimensional})
\begin{align*}
  \G_{\le n}\dparenl \R^d\dparenr = \{ \proj_{\le n} \IIS(\CURVE) \mid \text{$\CURVE:[0,T]\to\R^d$ is rectifiable} \} \subset
	T_{\le n}\dparenl\R^d\dparenr.
\end{align*}

Equivalently, the log-signature completely fills out the Lie algebra \(\g\dparenl\R^d\dparenr\). We have
\begin{equation*}
    \log\IIS(\CURVE)=\sum_{h\in\lyndon{d}} c_{h}(\CURVE)\,b_h,
\end{equation*}
where $c_{h}(\CURVE)=\langle\ISS(\CURVE),\zeta_h\rangle$ for uniquely determined $\zeta_h\in T(\R^d)$. This inspires us to also denote the coordinates of an arbitary $\ccn\in\lie{n}d$ by $c_{h}$\nomenclature[ch]{$c_{h}$}{The coordinate of $\ccn\in\lie{n}d$ corresponding to the Hall basis element $b_h$}, were analogously
\begin{equation*}
  \ccn=\sum_{\substack{h\in\lyndon{d},\\|h|\leq n}} c_{h}\,b_h.
\end{equation*}

\begin{example}[Moment curve]
  \label{ex:momentCurve}

  We consider the \textbf{moment curve} in dimension $3$,
  which is the curve $\CURVE: [0,1] \to \R^3$ given as
  \begin{align*}
    X_t \coloneq (t,t^2,t^3).
  \end{align*}
  It traces out part of the twisted cubic \cite[Example 1.10]{harris2013algebraic},
  see also \cite[Sect. 15]{karlin1953geometry}.
  
  We calculate, as an example,
  \begin{align*}
    \langle \ISS(\CURVE), \w{3}\w{2} \rangle
    &=
    \int_0^1 \int_0^s 3 r^2 dr 2 s ds \\
    &=
    2 \int_0^1 s^4 ds
    =
    \frac{2}{5}.
  \end{align*}

  The entire step-$2$ truncated signature is
  \begin{align*}
    \proj_{\le 2}\IIS(\CURVE)
    =
		\left(
      \begin{bmatrix}1\\1\\1\end{bmatrix},
    \begin{bmatrix}\frac12&\frac46&\frac34\\\frac26&\frac12&\frac6{10}\\\frac14&\frac4{10}&\frac12\end{bmatrix}\right),
  \end{align*}
  and the step-$2$ truncated log-signature is
	\[
    \proj_{\le 2}\log\IIS(\CURVE)
    =
		\left(\begin{bmatrix}1\\1\\1\end{bmatrix},\begin{bmatrix}0&\frac16&\frac14\\-\frac16&0&\frac1{10}\\-\frac14&-\frac1{10}&0\end{bmatrix}\right).
	\]
\end{example}

\subsection{Invariants}
\label{ss:invariants}

In this work we are interested in functions on paths that factor through
the signature and that are invariant
to a group $G$ acting on the path's ambient space $\R^d$.
We will mainly focus on $G = \O_d(\R)$, acting linearly on $\R^d$.
The action of \(A\in G\) on an \(\R^d\)-valued path \(\CURVE\) is given by \(AZ\colon[0,1]\to\R^d\), \(t\mapsto AZ_t\).

Using \Cref{prp:fext}, we can extend the action of $G$ on $\R^d$ to a \textbf{diagonal action} on words.
The matrix \(A^\top\) acts on single letters by
\begin{equation*}
	\phi_{A^\top}(\w{i})=\sum_j a_{ji} \w{j},
\end{equation*}
and we set \(\phi_{A^{\top}}(w)=0\) whenever \(|w|\ge 2\).
By \Cref{prp:fext}, this induces an endomorphism \(\tilde\phi_{A^\top}\colon T(\R^d)\to
T(\R^d)\), satisfying
\begin{equation}
	\tilde\phi_{A^\top}(\w w_1\dotsm\w w_n)=\phi_{A^\top}(\w w_1)\dotsm\phi_{A^\top}(\w w_n).
	\label{eq:phiconc}
\end{equation}
In particular, \(\tilde\phi_{A^\top}(u\w i)=\tilde\phi_{A^\top}(u)\tilde\phi_{A^\top}(\w i)\) for
all words \(u\) and letters \(\w i\in\{\w1,\dotsc,\w d\}\).
In order to be consistent with the notation in \cite{DR2019}, we will denote its transpose map ($\Phi_A$ in \Cref{prp:fext}) just
by \(A\colon T\dparenl\R^d\dparenr\to T\dparenl\R^d\dparenr\).
\begin{lemma}
	The map \(\tilde\phi_{A^\top}\colon T(\R^d)\to T(\R^d)\) is a shuffle morphism, that is,
	\[
		\tilde\phi_{A^\top}(u\shuffle v)=\tilde\phi_{A^\top}(u)\shuffle\tilde\phi_{A^\top}(v)
	\]
	for all words \(u,v\).
	\label{lem:Ashuffle}
\end{lemma}
\begin{proof}
	We proceed by induction on \(|u|+|v|\ge 0\).
	If \(|u|+|v|=0\) then necessarily \(u=v=\e\), and the identity becomes
	\[
		\tilde\phi_{A^\top}(\e\shuffle\e)=\tilde\phi_{A^\top}(\e)=\e=\e\shuffle\e=\tilde\phi_{A^\top}(\e)\shuffle\tilde\phi_{A^\top}(\e),
	\]
	which is true by definition.
	Now, suppose that the identity is true for all words \(u',v'\) with \(|u'|+|v'|<n\).
	If \(|u|+|v|=n\) we suppose, without loss of generality, that \(u=u'\w i, v=v'\w j\) for some
	(possibly empty) words \(u',v'\) with \(|u'|+|v'|<n\).
	Then
	\begin{align*}
		\tilde\phi_{A^\top}(u\shuffle v)&= \tilde\phi_{A^\top}(u'\w i\shuffle v'\w j)\\
		&= \tilde\phi_{A^\top}(u'\shuffle v'\w
		j)\tilde\phi_{A^\top}(\w i)+\tilde\phi_{A^\top}(u'\w i\shuffle v')\tilde\phi_{A^\top}(\w j)\\
		&= (\tilde\phi_{A^\top}(u')\shuffle\tilde\phi_{A^\top}(v'\w j))\tilde\phi_{A^\top}(\w
		i)+(\tilde\phi_{A^\top}(u'\w i)\shuffle\tilde\phi_{A^\top}(v'))\tilde\phi_{A^\top}(\w j)\\
		&= \tilde\phi_{A^\top}(u'\w i)\shuffle\tilde\phi_{A^\top}(v'\w j)\\
		&= \tilde\phi_{A^\top}(u)\shuffle\tilde\phi_{A^\top}(v).\qedhere
	\end{align*}
\end{proof}
\begin{remark}
	\Cref{lem:Ashuffle} is a special case of \cite[Theorem 1.2]{CP2018}.
\end{remark}
\begin{corollary}
  Let $A \in G$.
  \begin{enumerate}
    \item 
      The character group is invariant under \(A\), that is, \(A\cdot\G\dparenl\R^d\dparenr\subset\G\dparenl\R^d\dparenr\).

    \item 
      The restriction of \(A\) to \(\g\dparenl\R^d\dparenr\) is a Lie endomorphism.
      In particular, the free Lie algebra is invariant under \(A\), that is, \(A\cdot\g\dparenl\R^d\dparenr\subset\g\dparenl\R^d\dparenr\).

    \item 
      $\log: \G\dparenl\R^d\dparenr \to \g\dparenl\R^d\dparenr$ is an equivariant map.
  \end{enumerate}
	\label{crl:equivariant}
\end{corollary}
\begin{proof}~
  \begin{enumerate}
    \item 
			Let \(F\in\G\dparenl\R^d\dparenr\), and \(u,v\) be words. Then
			\begin{align*}
				\langle A\cdot F,u\shuffle v\rangle&= \langle F,\tilde\phi_{A^\top}(u\shuffle v)\rangle\\
				&= \langle F,\tilde\phi_{A^\top}(u)\shuffle\tilde\phi_{A^\top}(v)\rangle\\
				&= \langle F,\tilde\phi_{A^\top}(u)\rangle\langle F,\tilde\phi_{A^\top}(v)\rangle\\
				&= \langle A\cdot F,u\rangle\langle A\cdot F,v\rangle,
			\end{align*}
			that is, \(A\cdot F\in\G\dparenl\R^d\dparenr\).

    \item 
	Since \(A\cdot(FG)=(A\cdot F)(A\cdot G)\), \(A\) is automatically a Lie
	morphism.
	Now we check that \(A\cdot F\in\mathfrak g\dparenl\R^d\dparenr\) whenever \(F\in\mathfrak
	g\dparenl\R^d\dparenr\).
	It is clear that \(\langle A\cdot F,\e\rangle=\langle F,\e\rangle=0\).
	Now, if \(u,v\) are non-empty words, then
	\begin{align*}
		\langle A\cdot F,u\shuffle v\rangle&= \langle F,\tilde\phi_{A^\top}(u\shuffle v)\rangle\\
		&= \langle F,\tilde\phi_{A^\top}(u)\shuffle\tilde\phi_{A^\top}(v)\rangle\\
		&= 0,
	\end{align*}
	i.e. \(A\cdot F\in\mathfrak g\dparenl\R^d\dparenr\).

	\item Let \(G\in\G\dparenl\R^d\dparenr\). Then, since \(A\cdot\varepsilon=\varepsilon\) we get
    \begin{align*}
      \log( A \cdot G ) &= \sum_{n=1}^\infty\frac{(-1)^{n-1}}{n}(A \cdot G-\varepsilon)^n \\
      &= A \cdot \sum_{n=1}^\infty\frac{(-1)^{n-1}}{n}(G-\varepsilon)^n \\
      &= A \cdot\log(G).\qedhere
    \end{align*}
  \end{enumerate}
\end{proof}

In particular, we easily see that (see also \cite[Lemma 3.3]{DR2019})
\begin{equation}
\label{eq:AdotIIS}
  \IIS( A\cdot \CURVE ) = A \cdot \IIS( \CURVE ).
\end{equation}

The same is true for the truncated versions, and
we note that, in the special case of $\lie{2}d$, under the isomorphism in \cref{eq:lie2isom}, the action has the simple form
\begin{align}
  \label{eq:simpleForm}
  A\cdot (v,M) = (A v, A M A^\top),
\end{align}
where the operations on the right-hand side are matrix-vector resp. matrix-matrix multiplication.
It follows from \Cref{crl:equivariant} and \eqref{eq:AdotIIS} that \(\log( \IIS(A \CURVE) ) = A\cdot \log( \IIS(\CURVE) )\).
As already remarked, $\log$ is a bijection (with inverse $\exp$).
To obtain invariant expressions in terms of $\IIS(\CURVE)$ it is hence
enough to obtain invariant expressions in terms of $\log(\IIS(\CURVE))$.
Going this route has the benefit
of \emph{working on a linear object}.
To be more specific, $\IIS(\CURVE)$ is, owing to the shuffle relation,
highly redundant.
As an example in $d=2$,
\begin{align*}
  \Big\langle \IIS(\CURVE), \w{1} \Big\rangle^2 + \Big\langle \IIS(\CURVE), \w{2} \Big\rangle^2
  =
  2\ \Big\langle \IIS(\CURVE), \w{11} + \w{22} \Big\rangle.
\end{align*}
Now, both of these expressions are invariant to $\O_2(\R)$.
The left-hand-side is a nonlinear expressions in the signature, whereas the right-hand-side is a linear one.
To not have to deal with this kind of redundancy we work with the log-signature.
We note that in \cite{DR2019} the \emph{linear} invariants of the signature itself are presented.
Owing to the shuffle relation, this automatically yields (all) polynomial invariants.
But, as just mentioned, it also yields a lot of redundant information.

\begin{example}
  \label{ex:momentCurveLinearVsNonlinearInvariants}
  We continue with \Cref{ex:momentCurve}.
  The rotation
  \begin{align*}
    A =
    \begin{bmatrix}
      0 & 1 & 0 \\
      0 & 0 & 1 \\
      1 & 0 & 0
    \end{bmatrix},
  \end{align*}
  results in the curve
  \begin{align*}
    Y_t \coloneq A X_t =
    \begin{bmatrix}
      t^2 \\
      t^3 \\
      t 
    \end{bmatrix}.
  \end{align*}

  Its step-$2$ truncated signature is
  \begin{align*}
    \proj_{\le 2}\IIS(Y)
    =
		\left(
      \begin{bmatrix}1\\1\\1\end{bmatrix},
      \begin{bmatrix}
        \frac12 & \frac35 & \frac13 \\
        \frac25 & \frac12 & \frac14 \\
        \frac23 & \frac34 & \frac12
      \end{bmatrix}
    \right)
    =
		\left(
      A
      \begin{bmatrix}1\\1\\1\end{bmatrix},
      A
    \begin{bmatrix}\frac12&\frac46&\frac34\\\frac26&\frac12&\frac6{10}\\\frac14&\frac4{10}&\frac12 \end{bmatrix}
    A^\top
    \right)
    =
    A \cdot\proj_{\le 2}\IIS(\CURVE).
  \end{align*}
  The step-$2$ truncated log-signature is
  \begin{align*}
    \proj_{\le 2}\log\IIS(Y)
    =
		\left(\begin{bmatrix}1\\1\\1\end{bmatrix},
  \begin{bmatrix}
    0 & \frac{1}{10} & -\frac16 \\
    -\frac1{10} &    0  & -\frac14 \\
    \frac{1}{6} &  \frac14   &  0
   \end{bmatrix}\right)
    =
    A \cdot \proj_{\le 2}\log\IIS(\CURVE).
  \end{align*}

\end{example}

In the present work, we consider general, \emph{nonlinear} expressions of the {log-signature}.
That way, we use the economical form of the log-signature, while still providing a complete -- in a
precise sense -- set of nonlinear invariants.

\subsection{Moving frame method}

We now provide a brief introduction to the Fels-Olver moving frame method introduced in \cite{FO99}, a modern generalization of the classical moving frame method formulated by Cartan \cite{Car1951}. For a comprehensive overview of the method and survey of many of its applications see \cite{O18,fels1998moving}. We will assume in this subsection that $G$ is a finite dimensional Lie group acting smoothly\footnote{Here actly smoothly means that the map defining the group action is a $C^\infty$ map. For our purposes there is no loss
in taking 'smooth' to mean '$C^\infty$'..} on an $m$-dimensional manifold $M$.

\begin{definition}\label{def:mf}
A \textbf{moving frame} for the action of $G$ on $M$ is a smooth
map $\rho:M \rightarrow G$ such that $\rho(g\cdot z) = \rho(z) \cdot g^{-1}$.
\end{definition}

In general one can define a moving frame as a smooth $G$-equivariant map $\rho:M\rightarrow G$. For simplicity we assume $G$ acts on itself by right multiplication; this is often referred to as a \textit{right} moving frame. A moving frame can be constructed through the use of a cross-section to the orbits of the action of $G$ on $M$.

\begin{definition}\label{def:xsection}
A \textbf{cross-section} for the action of $G$ on $M$ is a submanifold $\K\subset M$ such that $\K$ intersects each orbit transversally at a unique point.
\end{definition}

\begin{definition}\label{def:free}
The action of $G$ is \textbf{free} if the \textbf{stabilizer} $G_z$\nomenclature[Gz]{$G_z$}{The stabilizer of a point $z$, the largest subgroup of $G$ that keeps $z$ invariant} of any point $z\in M$ is trivial, i.e.
$$
G_z := \{ g\in G\, |\, g\cdot z=z\} = \{\id\},
$$
where $\id\in G$ denotes the identity transformation.
\end{definition}

The following result appears in much of the previous literature on moving frames (see, for instance, \cite[Thm. 2.4]{Olv01}).

\begin{theorem}\label{thm:mfxsection}
  Let $G$ be an action on $M$ and assume that
  \begin{enumerate}
    \item[($\ast$)]
    The action is free, and around each point $z\in M$ there exists arbitrarily small
    neighborhoods whose intersection with each orbit is path-wise connected.
  \end{enumerate}
 If $\K$ is a cross-section, then the map $\rho: M\rightarrow G$ defined by sending $z$ to the unique group element $g\in G$ such that $g\cdot z\in \K$ is a moving frame.
\end{theorem}

\begin{remark}
The equivariance of the map $\rho:M\rightarrow G$ such that $\rho(z)\cdot z\in \K$ can be seen from the fact that $\rho(z) \cdot z = \rho(g\cdot z) \cdot (g\cdot z)$ for any $g\in G$. Since $G$ is free this implies that $ \rho(z)= \rho(g\cdot z) \cdot g$, and hence $\rho$ satisfies Definition \ref{def:mf}.
\end{remark}

Similarly, in this setting, a moving frame $\rho$ specifies a cross-section defined by $\K = \{\rho(z)\cdot z \in M \}$. This construction can be interpreted as a way to assign a ``canonical form'' to points $z\in M$ under the action of $G$, thus producing invariant functions on $M$ under $G$.

\begin{definition}
Let $\rho: M\rightarrow G$ be a moving frame. The \textbf{invariantization} of a function $F:M\rightarrow \R$ with respect to $\rho$ is the invariant function $\iota(F)$ defined by
$$
\iota(F)(z) = F(\rho(z)\cdot z).
$$
\end{definition}

Given a moving frame $\rho$ and local coordinates $z=(z_1,\hdots,z_m)$ on $M$, the invariantization of the coordinate functions $\iota(z_1),\hdots , \iota(z_m)$ are the \textbf{fundamental invariants} associated with $\rho$. In particular we can compute $\iota(F)$ by
$$
\iota(F)(z_1,\hdots,z_m)=F(\iota(z_1),\hdots,\iota(z_m)).
$$
Since $\iota(I)(z)=I(z)$ for any invariant function $I$, the fundamental invariants provide a functionally generating set of invariants for the action of $G$ on $M$. In general, we will call a set of invariants $\mathfrak{I}=\{J_1,\ldots,J_m\}$ \textbf{fundamental} if it functionally generates all invariants, i.e. for any invariant $I$ there is a function $I'$ such that \begin{equation*}
 I(p)=I'(J_1(p),\ldots,J_m(p)).
\end{equation*}

Now, suppose further  that $G$ is an $r$-dimensional Lie group and that $\rho$ is the moving frame associated to a \textbf{coordinate cross-section} $\K$ defined by equations
$$
z_1=c_1, \hdots, z_r=c_r
$$
for some constants $c_1,\hdots, c_r$. Then the first $r$ fundamental invariants are the \textbf{phantom invariants} $c_1,\hdots, c_r$, while the remaining $m-r$ invariants $\{I_1,\hdots, I_{m-r}\}$ form a functionally independent generating set. In this case we can see that two points $z_1,z_2\in M$ lie in the same orbit if and only if 
$$
I_1(z_1)=I_1(z_2), \hdots, I_{s}(z_1)=I_{s}(z_2).
$$

\begin{example}\label{ex:SO2mf}
Consider the canonical action of $\SO_2(\R)$ on $\R^2\setminus\{(0,0)\}$. This action
satisfies the assumptions of \Cref{thm:mfxsection}
and a cross-section to the orbits is given by
$$
\K = \{(x,y)\, |\, x=0, y>0\}.
$$

The unique group element taking a point to the intersection of its orbit with $\K$ is the rotation (see \Cref{fig:so2})

$$
\rho(x,y) = \begin{bmatrix}
\frac{y}{\sqrt{x^2+y^2}} & \frac{-x}{\sqrt{x^2+y^2}}\\
 \frac{x}{\sqrt{x^2+y^2}} & \frac{y}{\sqrt{x^2+y^2}} 
\end{bmatrix}.
$$

\begin{figure}[!ht]
    \centering
    \includegraphics[width=0.6\textwidth]{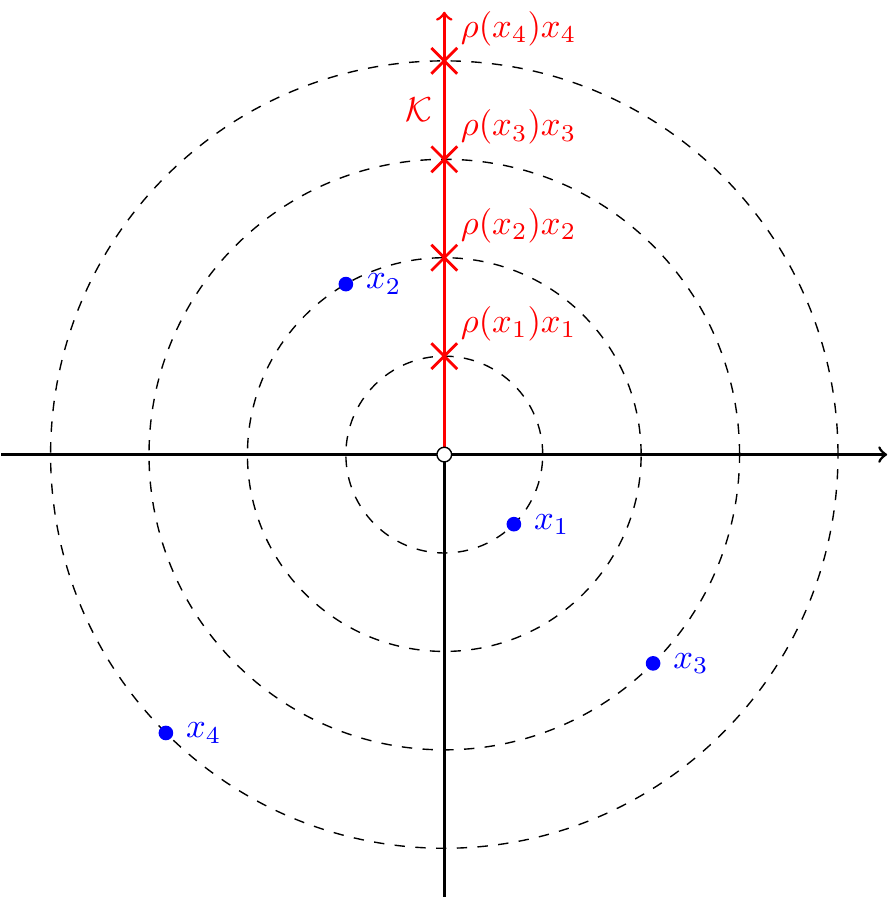}
    \caption{Cross section for the canonical action of the special orthogonal group \(\SO_2(\R)\).}
    \label{fig:so2}
\end{figure}

The fundamental invariants associated with the moving frame $\rho:\R^2\backslash\{(0,0)\}\rightarrow \SO_2(\R)$ are given by

$$
\iota(x) = 0\quad \iota(y)=\sqrt{x^2+y^2}.
$$

Thus any invariant function
for this action can be written as a function of $\iota(y)$, the Euclidean norm.
One can check that indeed for an invariant $I(x,y)$,
one has $I(x,y)=I(0,\sqrt{x^2+y^2})$.
This additionally implies that two points are related by a rotation if and only if they have the same Euclidean norm.
\end{example}

In practice it is difficult, or impossible, to find a global cross-section, and thus a global moving frame, to the orbits of $G$ on $M$. For instance in the above example, the origin was removed from $\R^2$ to ensure freeness of the action.
If the action of $G$ on $M$ satisfies condition ($\ast$) from \Cref{thm:mfxsection},
then the existence of a \textbf{local moving frame} around each point $z\in M$
is guaranteed by \cite[Thm. 4.4]{FO99}.
In this case the moving frame is a map $\rho: U\rightarrow V$ from a neighborhood $z\in U$ of $M$ to a neighborhood of the identity in $V\subset G$. The fundamental set of invariants produced are also local in nature and thus only guaranteed to be invariant on $U$ for elements $g\in V$.

The condition ($\ast$) in  \Cref{thm:mfxsection} can be relaxed in certain cases. In \cite[Sec. 1]{HK07b} the authors outline a method to construct a fundamental set of local invariants for actions of $G$ that are only semi-regular, meaning that all orbits have the same dimension. In particular Theorem 1.6 in \cite{HK07b} states that for a semi-regular action of $G$ on $M$, there exists a \textit{local} coordinate cross-section about every point $z\in M$. In a neighborhood $U$ containing $z$, such a linear space intersects transversally the connected component containing $\overline{z}$ of the orbit $G\cdot \overline{z}$ at a unique point for each $\overline{z}\in U$ and is of complementary dimension to the orbits of the action.

\begin{remark}\label{rem:AlgRegularity}
 The algebraic actions that we define in the next section are automatically semi-regular on a Zariski-open subset of the target space (Proposition \ref{prop:alggroupfacts}(c)), and hence a local cross-section exists around any point in this subset. Since orbits are algebraic subsets, a local coordinate cross-section is a submanifold of complementary dimension (to the dimension of orbits) intersecting each orbit about $z$ transversally, and hence in finitely-many points. If every sufficiently small neighborhood about $z$ does \textit{not} have path-wise connected intersection with each orbit, a local cross-section about $z$ necessarily intersects some orbit at infinitely-many points, and hence a free algebraic group action necessarily satisfies condition ($\ast$) from \Cref{thm:mfxsection}.
\end{remark}

\subsection{Algebraic groups and Invariants}\label{ssec:alggroups}

In this work, we will be in the setting of an algebraic group $G$ acting rationally on a variety
$X$. In other words $G$ is an algebraic variety equipped with a group structure, and the action of
$G$ on $X$ is given by a rational map $\Phi: G\times X \to X$. Here we outline some key facts and results about algebraic group actions and the invariants of such actions, following \cite{PV1994} for much of our exposition. Unless specified otherwise, both $G$ and $X$ are both varieties over the algebraically closed field $\C$.

The orbit $G \cdot p$ of a point $p\in X$ under $G$ is the image of $G\times \{p\}$ under the rational map $\Phi$ defining the action, and hence is open in its closure $\overline{G\cdot p}$ under the \textbf{Zariski topology}.%
\footnote{
The Zariski topology on an affine space $k^d$ is the topology where closed sets are given by subsets of the form $V(f_1,\hdots,f_s) = \{ (x_1,\hdots,x_d)\in k^d\,| \, f_1(x_1,\hdots,x_d)=\hdots=f_s(x_1,\hdots,x_d)=0\}$ for some collection of polynomials $f_1,\hdots,f_s\in k[x_1,\hdots,x_d]$. 
}

The following proposition summarizes a few basic results on orbits of algebraic groups that can be found in \cite[Section 1.3]{PV1994}.

\begin{proposition}\label{prop:alggroupfacts}
For any point $p\in X$, the stabilizer $G_p$ is an algebraic subgroup of $G$ and $G\cdot p$ satisfies the following:

\begin{itemize}
\item[(a)] The orbit $G\cdot p$ is a smooth, Zariski-open subset of $\overline{G\cdot p}$.

\item[(b)] The dimension of $G\cdot p$ satisfies $\dim G\cdot p = \dim G - \dim G_p,$ where $\dim G_p=\dim T_p (G\cdot p)$.

\item[(c)] The dimension of $G\cdot p$ is maximal on a non-empty Zariski-open subset of $X$.

\end{itemize}
\end{proposition}

For an arbitrary field $k$, we denote the ring of polynomial functions on the variety $X$ as $k[X]$, i.e. if $\mathcal{I}(X)$ is the ideal generated by the polynomials defining the variety $X\subset \mathbb{C}^d$, then $k[X] = k[x_1,x_2,\dotsc,x_d]/\mathcal{I}(X)$.\nomenclature[kX]{$k[X]$}{The ring of polynomial functions on the variety $X$ with coefficients in $k$} If $X$ is irreducible, then the field $k(X)$\nomenclature[k(X)]{$k(X)$}{The field of rational functions on the variety $X$ with coefficients in $k$} of rational functions on $X$ is defined similarly. The polynomial invariants (for the action of $G$ on the variety $X$) form a subring of $k[X]$ defined by
$$
k[X]^G = \{ f\in k[X]\, |\, f(g\cdot p) = f(p), \quad \text{for all}\quad g\in G, p\in X\}
$$
\nomenclature[kXG]{$k[X]^G$}{The subring of $k[X]$ of polynomial invariants for the action of $G$ on $X$}
and
the rational invariants form a subfield of $k(X)$ given by
$$
k(X)^G = \{ f\in k(X)\, |\, f(g\cdot p) = f(p), \quad \text{for all}\quad g\in G, p\in X\}
$$
\nomenclature[k(X)G]{$k(X)^G$}{The subfield of $k(X)$ of rational invariants for the action of $G$ on $X$}
respectively. Constructing invariant functions and finding generating%
\footnote{By a \textbf{generating set} for $k[X]^G$, we refer to a subset of $k[X]^G$ that generates $k[X]^G$ as a polynomial ring. Similarly a generating set of $k(X)^G$ is a subset that generates $k(X)^G$ as a field.} sets for $\C[X]^G$ is the subject of classical invariant theory \cite{Lit65, Olv99, Stu2008}. In \cite{Hilbert1890} Hilbert proved his finiteness theorem, showing that for linearly reductive groups acting on a vector space $V$ the polynomial ring $\C[V]^G$ is finitely generated leading him to conjecture in his fourteenth problem that $\C[X]^G$ is always finitely generated. In \cite{Nagata1959} Nagata constructed a counter-example to this conjecture.
For $\C(X)^G$, however, a finite generating set always exists and can be explicitly constructed (see for instance \cite{DK2015,HK07a}). Furthermore a set of rational invariants is generating if and only if it is also \textit{separating}.

\begin{definition}\label{def:GenSep}
A set of rational invariants $\INV\subset \C(X)^G$ \textbf{separates orbits on a subset $U\subset X$} if two points $p,q\in U$ lie in the same orbit if and only if $K(p)=K(q)$ for all $K\in \INV$. If there exists a non-empty, Zariski-open subset $X$ where $\INV$ separates orbits then we say $\INV$ is \textbf{separating}.
\end{definition}

\begin{proposition}\label{prop:GenSep}
For the action of $G$ on $X$, the field $\C(X)^G$ is finitely generated over $\C$. Moreover a subset $\INV\subset \C(X)^G$ is generating if and only if it is separating.
\end{proposition}

\begin{proof}
The backward direction holds by \cite[Lem. 2.1]{PV1994}. By \cite[Thm. 2.4]{PV1994} there always exists a finite set of separating invariants in $\C(X)^G$, and hence a finite generating set. Additionally this finite set can be rewritten in terms of any generating set, and hence any generating set is also separating.
\end{proof}

Under certain conditions the polynomial ring $\C[X]^G$ is also separating, as the following proposition from \cite[Prop. 3.4]{PV1994} shows.

\begin{proposition}
Suppose the variety $X$ is irreducible. There exists a finite, separating set of invariants $\INV\subset \C[X]^G$ if and only if $\C(X)^G=Q\C[X]^G$ where $Q\C[X]^G=\left.\left\{\frac{f}{g}\, \right|\, f,g\in \C[X]^G\right\}$.
\end{proposition}

One way to understand the structure of invariant rings is by considering subsets of $X$ that intersect a general orbit.

\begin{definition}\label{def:RelSec}
  Let $N \subset G$ be a subgroup.
A subvariety $S$ of $X$ is a \textbf{relative $N$-section} for the action of $G$ on $X$ if the following hold:
\begin{itemize}
\item 
  There exists a non-empty, $G$-invariant, and Zariski-open subset $U\subset X$, such that $S$ intersects each orbit that is contained in $U$.
  In other words, we have that $\overline{\Phi(G\times S)} = X$, where closure is taken in the Zariski topology.

\item One has  $N = \{ n\in G\, |\, nS=S\}$.
\end{itemize}
\end{definition}

We call the subgroup $N$ the \textbf{normalizer} subgroup of $S$ with respect to $G$. The following proposition summarizes a discussion in \cite[Sec. 2.8]{PV1994}.

\begin{example}
  For the action of $\SO_2(\mathbb{C})$ on the Zariski-open subset of $\C^2$ defined by $x^2+y^2\neq 0$, the variety $S$ defined by $x=0$ is a relative $N$-section for the action where $N$ is the 2-element subgroup generated by the reflection about the $y$-axis. Then $S$ intersects each orbit of the action in precisely two points.
\end{example}

\newcommand\restrict{\mathsf{R}}
\begin{proposition}\label{prop:NSec}
Let $S$ be a relative $N$-section for the action of $G$ on $X$.
Then the restriction map
\begin{align*}
  \restrict_{X\to S}\colon \C(X) \rightarrow \C(S),
\end{align*}
restricts to a field isomorphism between $\C(X)^G$ and $\C(S)^N$.
\end{proposition}

\begin{corollary}\label{cor:RelSecGen}
  Let $S$ be a relative $N$-section for the action of $G$ on $X$ and $\INV \subset \C(X)^G$ a set such that
  $\restrict_{X\to S}(\INV)$ generates $\C(S)^N$ where $\restrict_{X\to S}$ is the restriction map from \Cref{prop:NSec}.
  Then $\INV$ is a generating set for $\C(X)^G$.
\end{corollary}

Relative sections can be used to construct generating sets of rational invariants for algebraic actions as in \cite{GHP18}, which the authors refer to as the \textit{slice method}. Similar in spirit to the approach in \cite{HK07b}, considerations can be restricted to an algebraic subset of $X$. When the intersection of $S$ with each orbit is zero-dimensional, a relative $N$-section can be thought of as the algebraic analog to a \textit{local} cross-section for an action.

We end the section by considering algebraic actions on varieties defined over $\R$, where the issue
is more delicate. For instance, in this setting Proposition \ref{prop:GenSep} no longer holds meaning that generating sets of invariants are not necessarily separating and vice versa (see
\cite[Rem. 2.7]{KRV20}). 
Suppose that $X(\R)$\nomenclature[X(R)]{$X(\R)$}{A real variety with associated complex variety $X$} and $G(\R)$\nomenclature[G(R)]{$G(\R)$}{A real variety with associated complex variety $G$, of which it is also a subgroup} are real varieties with action given by
$\Phi: G(\R) \times X(\R) \to X(\R)$ and that $X$ and $G$ are the associated complex varieties. Then $\Phi$ defines an action of $G$ on $X$.

\begin{proposition}\label{prop:RealComplexInvFields}
$\R(X(\R))^{G(\R)}$ is a subfield of $\C(X)^G$.
\end{proposition}

\begin{proof}
If $f\in \R(X(\R))^{G(\R)}$, then the rational function $f(g\cdot p) - f(p)$ is identically zero on $G(\R)\times X(\R)$, and hence is identically zero on $G\times X$. Thus $f\in \C(X)^G$.
\end{proof}

\begin{corollary}\label{cor:RealGenComplex}
If $\INV=\{I_1,\hdots,I_s\}\subset \R(X(\R))^{G(\R)}$ generates $\C(X)^G$ then $\INV$ generates $\R(X(\R))^{G(\R)}$.
\end{corollary}

\begin{proof}
Suppose that $\INV$ generates $\C(X)^G$ and that $f\in  \R(X(\R))^{G(\R)}$. Then there exists a rational function $g\in \C(y_1,\hdots,y_s)$ such that $f = g(I_1,\hdots,I_s)$. We can decompose $g$ as $g= \text{Re}(g)+i\cdot\text{Im}(g)$ where $\text{Re}(g), \cdot\text{Im}(g)\in \R(y_1,\hdots,y_s)$. Since $f$ is a real rational function 

$$
2 f =  [\text{Re}(g)+i\cdot\text{Im}(g)] +  [\text{Re}(g)-i\cdot\text{Im}(g)] = 2 \text{Re}(g).
$$ 

Thus $g$ must lie in $\R(y_1,\hdots,y_s)$ proving the result.
\end{proof}

\begin{proposition}\label{prop:SepGenReal}
Suppose that $\R(X(\R))^{G(\R)}$ separates orbits for the action of $G(\R)$ on $X(\R)$. Then so does any generating set for $\R(X(\R))^{G(\R)}$.
\end{proposition}

\begin{proof}
Suppose that $\INV=\{I_1,I_2,\hdots\}$ generates $\R(X(\R))^{G(\R)}$ and that $\R(X(\R))^{G(\R)}$ separates orbits. Then for any two points $p_1, p_2\in X(\R)$ if
$$
I_1(p_1)=I_1(p_2), I_2(p_1)=I_2(p_2), \hdots
$$
for all invariants in $\INV$, then we also have $I(p_1)=I(p_2)$ for any invariant $I\in \R(X(\R))^{G(\R)}$ as $\INV$ generates $\R(X(\R))^{G(\R)}$. Thus $p_1$ and $p_2$ lie in the same orbit under $G(\R)$.
\end{proof}

\section{Orthogonal invariants on \texorpdfstring{$\lie{2}d$}{g2((Rd))}}\label{sec:complexdetour}

In this section we take a closer look at the action of $\O_d(\R)$ on $\lie{2}d \cong \RdSkew$.
In particular we construct an explicit linear space, of complementary dimension to the orbits, intersecting each orbit in a large open subset of this space.
To achieve this, we consider the associated action of the \textit{complex} group $\O_d(\C)$ on the space $\CdSkew$ where
$$
\O_d(\C) = \{A\in \GL_d(\C) \mid AA^T = \text{id}\}.
$$

As described in Section \ref{ssec:alggroups}, we can consider $\O_d(\R)$ and $\RdSkew$ as the real points of the varieties $\O_d(\C)$ and $\CdSkew$.
 
\begin{remark}
  The real Lie group 
  \begin{align*}
    \O_d(\R) \coloneq \{ A \in \R^{d\times d} : A A^\top = \id \},
  \end{align*}
  can be considered as a subgroup of the Lie group
  \begin{align*}
    \O_d(\C) \coloneq \{ A \in \C^{d\times d} : A A^\top = \id \}.
  \end{align*}
  We note that $\O_d(\C)$ is a complex Lie group, in
  contradistinction to the
  Lie group of unitary matrices
  \begin{align*}
    \operatorname{U}_d \coloneq \{ A \in \C^{d\times d}: A^* A = \id \},
  \end{align*}
  where $A^*$ is the conjugate transpose of $A$.
  Even though $\operatorname{U}_d$ contains matrices with complex entries,
  it is a real Lie group.
\end{remark}

By investigating the associated complex action, we can utilize tools such as the relative sections described in Definition \ref{def:RelSec}, and then pass these results down to the real points. As before in  \eqref{eq:simpleForm} the action of $\OdC$ on $\CdSkew$ is given by
\begin{equation}\label{eq:OdRActionLie2}
A\cdot (v,M) = (Av,AMA^T).
\end{equation}

We denote the entries as
$$
v = \begin{bmatrix}
    \vv_{1}\\
    \vv_{2}\\
    \vdots\\
    \vv_{ d}
\end{bmatrix}, \quad M = \begin{bmatrix}
0 & \m_{12} & \m_{13} & \hdots  &\m_{1d}\\
-\m_{12} & 0 & \m_{23} & \hdots & \m_{2d}\\
-\m_{13} & -\m_{23} & 0 & \hdots  & \vdots \\
\vdots & && \ddots & \m_{(d-1)d}\\
-\m_{1d} & -\m_{2d} & \hdots & -\m_{(d-1)d} &0
\end{bmatrix}
$$
to make explicit the connection to Section \ref{sec:MFsection}.

\begin{proposition}\label{prop:NonZeroNormRot}
For any $v\in \C^d$ such that $\vv_1^2+\dotsb+\vv_d^2\neq 0$, there exists $A\in \OdC$ such that
$\overline{v}=Av$ satisfies $\overline{\vv}_{1}=\dotsb=\overline{\vv}_{ d-1}=0$ and $\overline{\vv}_{d}\neq 0$.
\end{proposition}

\begin{proof}

The function $(d-1)(\vv_{1}^2+\dotsb+\vv_{d}^2)$ can be written as the sum of all pairwise sum of squares, i.e.

$$
(d-1)(\vv_{1}^2+\dotsb+\vv_{d}^2) = \sum_{i=1}^d\sum_{j\neq i} \left( \vv_{i}^2+\vv_{j}^2 \right).
$$

Suppose that $\vv_{1}^2+\dotsb+\vv_{d}^2\neq 0$ and that there exists some $\vv_{i}\neq 0$ where $1\leq i\leq d-1$. (Otherwise we are done by choosing $A$ as the identity.)
By the above equation, there exists a pair of coordinates $\vv_{i}$ and $\vv_{j}$ such that $\vv_{i}^2+\vv_{j}^2\neq 0$ for some $1\leq i < j\leq d$.

Choose the matrix $A \in \OdC$ defined by

$$
a_{k\ell}=\begin{dcases} 
      1 & k=\ell \neq i,j \\
     \frac{\vv_{j}}{w} & k=\ell=i,j \\
      -\frac{\vv_{i}}{w}& k=i, \ell=j\\
      \frac{\vv_{j}}{w} & k=j, \ell=i\\
      0 & \text{otherwise}
   \end{dcases}
$$
where $w$ is an element of $\C$ that satisfies $w^2 = \vv_i^2+\vv_j^2$. The transformation $A$ is the complex analogue to a Givens Rotation which only rotates two coordinates. Then for $Av = \overline{v}$ we have that $\overline{\vv}_k=\vv_k$ for $k\not\in \{ i,j \}$, $\overline{\vv}_i = 0$, and $\overline{\vv}_j = w \neq 0$. This process can be repeated until $\overline{v}$ is of the desired form.
\end{proof}

We define a sequence of linear subspaces of $\CdSkew$ as
\newcommand\LC[2]{{L_{#1}^{(#2)}}}
\begin{align}
\label{eq:Ls}
\begin{split}
\LC{d}1 &= \{ (v,M)\in \CdSkew \,|\, \vv_1=\cdots=\vv_{d-1}=0\},\\
\LC{d}i &= \{ (v,M)\in \LC{d}{i-1}\,|\, \m_{1(d-i+2)}=\cdots=\m_{(d-i)(d-i+2)}=0\}, \qquad 2 \leq i \leq d-1.
\end{split}
\end{align}
\nomenclature[Ldi]{$\LC{d}{i}$}{The relative $\NC{d}{d-i}$-section for the action of $\OdC$ on $\CdSkew$}
\newcommand\LL{\LC{d}{d-1}}
In particular the subspace $\LC{d}{d-1}$\nomenclature[Ld]{$\LL$}{The relative $\Wd$-section for the action of $\OdC$ on $\CdSkew$} is given by pairs $(v,M)$ of the form
\begin{equation}\label{eq:vMform}
v = \begin{bmatrix}
0\\
0\\
\vdots\\
\vv_d
\end{bmatrix}\quad M = \begin{bmatrix}
0 & \m_{12} & 0 & \hdots  &0\\
-\m_{12} & 0 & \m_{23} & \hdots & 0\\
0 & -\m_{23} & 0 & \hdots  & \vdots \\
\vdots & && \ddots & \m_{(d-1)d}\\
0 & 0 & \hdots & -\m_{(d-1)d} &0
\end{bmatrix}.
\end{equation}

\begin{example}
  For $d=4$,
  elements in $\LC{4}1$
  are of the form
  \begin{align*}
    (
    \begin{bmatrix}
      0  \\
      0  \\
      0  \\
      *  
    \end{bmatrix},
    \begin{bmatrix}
      0 &* &* &*  \\
      * &0 &* &*  \\
      * &* &0 &*  \\
      * &* &* &0 
    \end{bmatrix}),
  \end{align*}
  elements in $\LC{4}2$
  are of the form
  \begin{align*}
    (
    \begin{bmatrix}
      0  \\
      0  \\
      0  \\
      *  
    \end{bmatrix},
    \begin{bmatrix}
      0 &* &* &0  \\
      * &0 &* &0  \\
      0 &* &0 &*  \\
      0 &* &* &0 
  \end{bmatrix})
\end{align*}
and elements in $\LC{4}3$ are of the form
\begin{align*}
    (
    \begin{bmatrix}
      0  \\
      0  \\
      0  \\
      *  
    \end{bmatrix},
    \begin{bmatrix}
      0 &* &0 &0  \\
      * &0 &* &0  \\
      0 &* &0 &*  \\
      0 &0 &* &0 
  \end{bmatrix}).
  \end{align*}
  Note again that all $\Skew{d}{\C}$ matrices are skew-symmetric and thus have zero diagonal.
\end{example}

We will show that $\LC{d}1, \LC{d}2, ..$ form a sequence of relative sections for the action of $\OdC$ on $\CdSkew$ (see Definition \ref{def:RelSec}).
For this aim we need to identify the normalizer subgroup for each $\LC{d}i$,
which will be achieved in \Cref{prop:LiNi}.

The group $\O_i(\C)$, for $1\leq i < d$ appears as a subgroup of $\OdC$ in several natural ways, in particular the subgroup obtained by considering elements that rotate some fixed subset of $i$ coordinates and fix the remaining coordinates.
For $B\in \O_i(\C)$, denote
\begin{equation}\label{eq:SOimatrix}
  E(B)=\begin{bmatrix}
B & 0 &\cdots & 0\\
0 & 1 & \cdots &0\\
0 &  0 & \ddots & 0\\
0 & 0 & \cdots &1
\end{bmatrix},
\end{equation}
a matrix rotating the first $i$ coordinates and fixing the last $d-i$.
The set of such $E(B)$ forms a subgroup of $\OdC$ isomorphic to $\O_i(\C)$ which we will denote\nomenclature[OdiC]{$\OiC$}{The subgroup of $\OdC$ isomorphic to $\O_i(\C)$ which leaves the last $d-i$ components of a $\C^d$ vector invariant} 
\begin{align*}
 \OiC.
\end{align*}
Note that $\OiC \subset \O^{i+1}_d(\C)$.

\begin{proposition}\label{prop:Roti}
  Let $1\leq i < d$ and
  $B \in \O_i(\C)$.
  The image of the coordinates $\m_{1(i+1)}, \m_{2(i+1)},\hdots, \m_{i(i+1)}$ under the action of $E(B) \in\OiC$ on $(v,M) \in \CdSkew$ is given by
  \begin{align*}
    B
\begin{bmatrix}
\m_{1(i+1)}\\
\m_{2(i+1)}\\
\vdots \\
\m_{i(i+1)}
\end{bmatrix},
  \end{align*}
  the standard action of $\O_i(\C)$ on a vector in $\C^i$.
\end{proposition}

\begin{proof}
  This follows from \eqref{eq:OdRActionLie2}.
\end{proof}

Consider the subgroup 
\newcommand\Wd{{W_d(\C)}}\nomenclature[WdC]{$\Wd$}{The group of diagonal matrices with diagonal entries in $\{-1,1\}$; the normalizer of $\LL$}
\begin{align*}
  \Wd := \Big\{ \text{diagonal matrices with diagonal entries lying in $\{-1,1\}$}\Big\} \subset \OdC.
\end{align*}
\newcommand\NC[2]{{N_{#1}^{#2}(\C)}}
The action of an element of $\Wd$ changes the sign of various coordinates of $\CdSkew$. We define the subgroup $\NC{d}{i}$\nomenclature[NdiC]{$\NC{d}{i}$}{The product of the groups $\OiC$ and $\Wd$; the normalizer of $\LC{d}{d-i}$} of $\OiC$ as

$$
\NC{d}{i} := \OiC\cdot \Wd = \{ g\cdot h\, |\, g\in \OiC, h\in \Wd\}.
$$

Note that $\NC{d}{i}$ exactly contains matrices of the form
\begin{equation}
  \label{eq:ofTheForm}
\begin{bmatrix}
B & 0 &\cdots & 0\\
0 & \pm1 & \cdots &0\\
0 &  0 & \ddots & 0\\
0 & 0 & \cdots &\pm1
\end{bmatrix},
\end{equation}
with $B \in \OiC$.

\begin{proposition}\label{prop:LiNi}
  For $1\leq i < d$, the normalizer of $\LC{d}i$ is equal to $\NC{d}{d-i}$.
\end{proposition}

\begin{proof}

  It is immediate that $\NC{d}{d-1}$ leaves the space $\LC{d}1$ invariant.
  Considering 
  \begin{align*}
    x = (
    \begin{bmatrix}
      0 \\
      \dots \\
      0 \\
      1
    \end{bmatrix} ,M) \in \LC{d}1,
  \end{align*}
  we see that for $g \in \O_d(\C)$
  to have
  \begin{align*}
    g x \in \LC{d}1, 
  \end{align*}
  we must have $g_{j d} = g_{d j} = 0$, $j=1,\dots,d-1$.
  This proves the claim for $i=1$.

  Let the statement be true for some $1\le i \le d-2$.
  First, the normalizer
  of $\LC{d}{i+1}$ is contained in $\LC{d}i$.
  Diagonal entries of $\pm 1$ leave every $\LC{d}j$ invariant,
  so it remains to investigate the matrix $B$ in \eqref{eq:ofTheForm}.
  Now by \Cref{prop:Roti} $B$ acts by standard matrix multiplication
  on the vector $(\m_{1(i+1)}, \dots, \m_{i(i+1)})^\top$.
  We can hence apply the argument of the case $\LC{d}1$ to deduce
  that $\NC{d}{d-(i+1)}$ is the normalizer of $\LC{d}{i+1}$.

\end{proof}

We now show that $\LL$ is a relative $\Wd$-section, by constructing a sequence of relative sections for the action, drawing inspiration from recursive moving frame algorithms (see \cite{Kog03} for instance).

\newcommand\UC{{U_d(\C)}}

\begin{proposition}\label{prop:ComplexCS}
The linear space $\LL$ is a relative $\Wd$-section for the action of $\OdC$ on $\CdSkew$.
More precisely, there exists a set of rational invariants\nomenclature[Id]{$\INV_d$}{The set of rational invariants defining $\UC$}

\begin{equation}\label{eq:Id}
\INV_d = \{f_1, \hdots, f_{d}\} \subset \C(\CdSkew)^{\OdC}
\end{equation}
such that if we define the invariant, non-empty, Zariski-open subset\nomenclature[UdC]{$\UC$}{The Zariski-open subset of $\CdSkew$ where none of the invariants in $\INV_d$ vanishes}
\begin{equation}\label{eq:UC}
\UC = \left\{ (v,M)\in\CdSkew\,\left|\, 
(v,M) \text{ is in the domain of each $f_k$ and }
\prod_{k=1}^{d} f_k(v,M) \neq 0\right\}\right.
\end{equation}
we have that $\LL$ intersects each orbit that is contained in $\UC$.
Furthermore we can restrict each invariant to obtain
\begin{itemize}
\item $f_1 = \vv_1^2+\hdots+\vv_d^2$,
\item $f_i|_{\LC{d}{i-1}} = \m_{1(d-i+2)}^2+\hdots+\m_{(d-i+1)(d-i+2)}^2$ for $2\leq i <d$.
\item $f_d|_{\LL} = \m_{12}^2$.
\end{itemize}
\end{proposition}

\begin{proof}
By Proposition \ref{prop:NonZeroNormRot}, outside of $f_1= ||v||^2=0$, there exists a rotation
$A_1\in\OdC$ such that $A_1\cdot (v,M)\in \LC{d}1$. Thus, by \Cref{prop:LiNi}, $\LC{d}1$ is a relative $\NC{d}{d-1}$-section. We also have that $f_1|_{\LC{d}1} = \vv_d^2$. We proceed by induction. Suppose that for each point in $U_i = \{ \prod_{k=1}^{i} f_k(p) \neq 0\}$ there exists a rotation $B_{i}\in \OdC$ such that $B_{i}\cdot (v,M) \in \LC{d}i$.

By Proposition \ref{prop:LiNi}, the linear space $\LC{d}i$ is a relative $\NC{d}{d-i}$-section and, by Proposition \ref{prop:NSec}, there exists a field isomorphism 
$\sigma_i : \C(\CdSkew)^{\OdC} \rightarrow \C(\LC{d}i)^{\NC{d}{d-i}}$\nomenclature[sigmai]{$\sigma_i$}{The field isomorphism $\sigma_i : \C(\CdSkew)^{\OdC} \rightarrow \C(\LC{d}i)^{\NC{d}{d-i}}$}. Using proposition \ref{prop:Roti}, one can show that on $\LC{d}i$ the polynomial $\m_{1(d-i+2)}^2+\hdots+\m_{(d-i+1)(d-i+2)}^2$ lies in $\C(\LC{d}i)^{\NC{d}{d-i}}$. Let $f_{i+1}$ be the unique element in $\C(\CdSkew)^{\OdC}$ such that $f_{i+1}= \sigma_i^{-1}( \m_{1(d-i+2)}^2+\hdots+\m_{(d-i+1)(d-i+2)}^2)$.

By Proposition \ref{prop:NonZeroNormRot}, for any $(v,M)\in \LC{d}i$ outside of $\{f_{i+1}(v,M)=0\}$, there exists a rotation $A_{i+1}\in \NC{d}{d-i}$ such that $A_{i+1}\cdot (v,M) \in \LC{d}{i+1}$. Thus for any $(v,M)$ in $U_{i+1} = \{ \prod_{k=1}^{i+1} f_k(v,M) \neq 0\}$ there exists a rotation $B_{i+1}=A_{i+1}B_i\in \OdC$ such that $B_{i+1}\cdot (v,M) \in \LC{d}{i+1}$. Using Proposition \ref{prop:LiNi}, again, we see that $\LC{d}{i+1}$ is a relative $\NC{d}{d-i-1}$-section.

We can continue this induction until we have $f_{d-1}$ where $f_{d-1}|_{\LC{d}{d-2}} = \m_{13}^2+\m_{23}^2$. Finally note that the polynomial $\m_{12}^2$ lies in $ \C(\LL)^\Wd$. 
Since $\LL$ is a $\Wd$-section
(since $\Wd = \NC{d}1$)
there exists $f_d\in \C(\CdSkew)^{\OdC}$ such that $f_d|_{\LL} = \m_{12}^2$.
\end{proof}

\begin{remark}
 Denoting $\varsigma_1:=\sigma_1,\,\varsigma_{i+1}:=\sigma_{i+1}\circ\sigma_i^{-1}$, we have the following chain of $O_d(\R)$ transformations $A_i$ and field isomorphisms $\varsigma_i$:
 \begin{center}
 \begin{tikzcd}
 \CdSkew \arrow[r, "A_1"] &
 \LC{d}1 \arrow[r, "A_2"] &
 \ldots \arrow[r, "A_{d-2}"] &
 \LC{d}{d-2} \arrow[r, "A_{d-1}"] &
 \LC{d}{d-1}
 \\
  \C(\CdSkew)^{\OdC} \arrow[r, "\varsigma_1"] & 
  \C(\LC{d}1)^{\NC{d}{d-1}} \arrow[r, "\varsigma_2"] &
  \ldots \arrow[r, "\varsigma_{d-2}"] &
  \C(\LC{d}{d-2})^{\NC{d}{2}} \arrow[r, "\varsigma_{d-1}"] &
  \C(\LC{d}{d-1})^{\Wd}
  \end{tikzcd}
  \end{center}
  Note though that while the $\varsigma_i$ are uniquely determined, the $A_i$ are not. The composition $A_{d-1}A_{d-2}\cdots A_{2}A_{1}$ however is unique up to a multiplication of a $\Wd$ matrix from the left.
\end{remark}
In particular the above proposition implies that $\LL$ is a relative $\Wd$-section for the action of $\OdC$ on $\CdSkew$, and hence the function fields $\C(\LL)^\Wd$ and $\C(\CdSkew)^{\OdC}$ are isomorphic. By examining the action of $\Wd$ on $\LL$ and the structure of $\C(\LL)^\Wd$ we can therefore glean information about the action of $\OdC$ on $\CdSkew$. Consider a diagonal matrix $D\in \Wd$ given by

$$
D = \begin{bmatrix}
w_1 & 0 &\hdots & 0\\
0 & w_2 & \hdots & 0\\
\vdots &  & \ddots &\vdots\\
0 & 0 & \hdots &w_d
\end{bmatrix}
$$
where $w_i \in \{-1,1\}$ for $1\leq i\leq d $. Then the image of a point in $L_d^{(d-1)}$ is $D\cdot (v, M)=(\overline{v}, \overline{M})$ where
\begin{center}
\resizebox{.95\linewidth}{!}{
  \begin{minipage}{\linewidth}
\begin{equation}\label{eq:Waction}
\overline{v} = \begin{bmatrix}
0\\
0\\
\vdots\\
w_d\vv_d
\end{bmatrix}\quad \overline{M} = \begin{bmatrix}
0 & w_1w_2\m_{12} & 0 & \hdots  &0\\
-w_1w_2\m_{12} & 0 & w_2w_3\m_{23} & \hdots & 0\\
0 & -w_2w_3\m_{23} & 0 & \hdots  & \vdots \\
\vdots & && \ddots & w_{d-1}w_d\m_{(d-1)d}\\
0 & 0 & \hdots & -w_{d-1}w_d\m_{(d-1)d} &0
\end{bmatrix}.
\end{equation}
\end{minipage}
}
\end{center}

\begin{proposition}\label{prop:WFree}
The action of $\Wd$ on $\LL\cap \UC$ is free.
\end{proposition}

\begin{proof}
Suppose that the action is not free. Then there exists $D\in \Wd$ such that $D\cdot (v,M) = (v,M)$ and $D$ is not the identity. Necessarily we have that for some $1\leq i\leq d-1$, $w_i=-1$. Since $w_iw_{i+1}\m_{i(i+1)} = \m_{i(i+1)}$ and $\m_{i(i+1)}\neq 0$, then $w_{i+1}=-1$. Using a similar argument, $w_{i+2}=-1$ and so forth. However $w_d\vv_d=\vv_d$, where $\vv_d\neq 0$, implying that $w_d=1$ which is a contradiction.
\end{proof}

\begin{corollary}\label{cor:OFree}
The action of $\OdC$ on $\UC\subset \CdSkew$ is free.
\end{corollary}

\begin{proof}
By Proposition \ref{prop:ComplexCS}, each orbit on $\UC$ meets the linear subspace $\LL$. We show that the stabilizer of a point in $\LL\cap \UC$ contains only the identity. This is sufficient to prove the result, as any two points in the same orbit have isomorphic stabilizer groups.

Let $(v,M)\in \LL$ and consider $g\in G$ such that $g\cdot (v,M)=(v,M)$. By the proof of Proposition \ref{prop:LiNi} $g$ must lie in $\Wd$. However, by Proposition \ref{prop:WFree}, the only element of $\Wd$ fixing a point in $\LL\cap \UC$ is the identity.
\end{proof}

Since we have that $w_i^2=1$ for any $1\leq i \leq d$, clearly \nomenclature[IWdC]{$\INV_\Wd$}{The generating set for $\C(\LL)^\Wd$ given by $\sigma_{d-1}(\INV_d)$}
\begin{equation}\label{eq:WInv}
\INV_\Wd := \{ \vv_d^2, \m_{d(d-1)}^2,\hdots, \m_{12}^2\}
\end{equation}
is a set of invariant functions on $\LL$. 

\begin{proposition}\label{prop:WInvGen}
The set $\INV_\Wd$ separates orbits and is a generating set for $\C(\LL)^\Wd$.
\end{proposition}

\begin{proof}
Consider the map $F: \LL\cap \UC \rightarrow \C^d$ defined by evaluating the invariants in $\INV_\Wd$ on $\LL\cap \UC$, a non-empty, Zariski-open subset of $\LL$. We show that every fiber of this map is exactly an orbit of $\Wd$. Consider any $(v,M)\in \LL\cap \UC$; then set of points in the fiber of its image is given by

\begin{align*}
F^{-1}(F(v,M)) &= \{ (\tilde{v},\tilde{M}) \in \LL\cap \UC \, |\, \tilde{\vv}_d^2=\vv_d^2, \tilde{\m}_{12}^2 =\m_{12}^2, \hdots, \tilde{\m}_{(d-1)d}^2=\m_{(d-1)d}^2\}\\
&=\{(\tilde{v},\tilde{M}) \in \LL\cap \UC \, |\, \tilde{\vv}_d = \pm \vv_d, \tilde{\m}_{12} = \pm \m_{12}, \hdots, \tilde{\m}_{(d-1)d}=\pm \m_{(d-1)d} \}.
\end{align*}

We can individually change the sign for any coordinate of $(v,M)$. To change the sign of only $\vv_d$ one can act by the matrix $D\in \Wd$ such that $w_i=-1$ for all $1\leq i \leq d$. Similarly for $\m_{i(i+1)}$ we can act by the matrix $D\in \Wd$ such that $w_k = -1$ for $1\leq k \leq i$ and $w_k =1 $ otherwise. This implies that the above set is exactly the orbit of $(v,M)$ under $\Wd$, and hence $\INV_\Wd$ is separating on $\LL\cap \UC$.
Then by Proposition \ref{prop:GenSep}, $\INV_\Wd$ generates $\C(\LL)^W$.
\end{proof}

\begin{corollary}
The set $\INV_d$ in \eqref{eq:Id} is a minimal generating set of rational invariant functions for $\C(\CdSkew)^{\OdC}$ and separates orbits.
\end{corollary}

\begin{proof}
By Proposition \ref{prop:ComplexCS},  $\LL$ is a relative $\Wd$-section for the action of $\OdC$ on $\CdSkew$, and $\INV_d$ restricts to the set of invariants $\INV_\Wd$ in \eqref{eq:WInv} for the action of $\Wd$ on $\LL$. 
This means $\INV_\Wd=\sigma_{d-1}(\INV_d)$, where $\sigma_{d-1}$ is the isomorphism from the proof of Proposition \ref{prop:ComplexCS}. 
By Proposition \ref{prop:WInvGen}, the set $\INV_\Wd$ is a generating set for $\C(\LL)^\Wd$, and hence by Corollary \ref{cor:RelSecGen}, $\INV_d$ is a generating set for $\C(\CdSkew)^{\OdC}$. By Proposition \ref{prop:GenSep}, $\INV_d$ also separates orbits.

By Corollary \ref{cor:OFree}, the action of $\OdC$ is free on a non-empty, Zariski-open subset of $\CdSkew$. Thus the maximum dimension of an orbit on $\CdSkew$ is $\dim(\OdC)=\frac{d(d-1)}{2}$. By \cite[Corollary, Section 2.3]{PV1994} the transcendence degree\footnote{The transcendence degree of $\C(X)^G$ is given by the cardinality of the largest set $\{f_1,\hdots,f_n\}\in \C(X)^G$ such that there does not exist a rational function $F$ where $F(f_1,\hdots,f_n)\equiv 0\in \C(X)^G$.} 
of $\C(\CdSkew)^{\OdC}$ is $\frac{d(d+1)}{2} -  \frac{d(d-1)}{2} =d$, and hence any generating set must be at least of size $d$, implying that $\INV_d$ is minimal.
\end{proof}

The above results for the action of $\OdC$ on $\CdSkew$ help uncover the structure of the action of $\OdR$ on $\RdSkew$. First we show that the intersection of the set $\UC$ defined in \eqref{eq:UC} with $\RdSkew$ is a non-empty and well-defined Zariski open subset of $\RdSkew$.

\newcommand\LLR{L_d^{(d-1);\R}}
\newcommand\UR{{U_{d}(\R)}}
\begin{proposition}\label{prop:RealCS}
The set $\INV_d$ in \eqref{eq:Id} is a subset of $\R(\RdSkew)^{\OdR}$. In particular \nomenclature[UdR]{$\UR$}{The intersection of $\UC$ and $\RdSkew$, a Zariski open subset of $\RdSkew$}
\begin{align*}
  \UR := \UC \cap \left[\RdSkew\right],
\end{align*}
is a well-defined, $\OdR$-invariant, and non-empty Zariski open subset of $\RdSkew$, and\nomenclature[Ldd-1R]{$\LLR$}{The intersection of $\LL$ and $\RdSkew$}
\begin{align*}
  \LLR :=\LL \cap \left[\RdSkew\right] 
\end{align*}
intersects each orbit (under $\OdR$) contained in $\UR$. 
\end{proposition}

\begin{proof}
In the proof of Proposition \ref{prop:ComplexCS}, each function $f_i$ is obtaining by taking the inverse image of a real invariant function under the field isomorphism $\sigma_i : \C(\CdSkew)^{\OdC} \rightarrow \C(\LC{d}i)^{\NC{d}{d-i}}$. The function $f_i$ can be decomposed $f_i = h_1 + I\cdot h_2$, where $h_1$ and $h_2$ are elements of $\R(\RdSkew)^{\OdR}$, and hence by Proposition \ref{prop:RealComplexInvFields}, are elements of $\C(\CdSkew)^{\OdC}$. Thus $h_1|_{\LC{d}i}= f_i|_{\LC{d}i}$. Since $\sigma_i$ is a field isomorphism, $f_i$ must define the same rational function as $h_1$, and hence is an element of $\R(\RdSkew)^{\OdR}$.

Note that Proposition \ref{prop:NonZeroNormRot} also holds for any $v\in\R^d$, i.e. by applying Gram-Schmidt to a linearly independent set of $d$ vectors $\{v,v_1,\hdots,v_{d-1}\}$ in $\R^d$. Thus if $f_1(v,M)\neq 0$, there exists a rotation $A\in \OdR$ such that $A\cdot (v,M)\in \LC{d}1\cap\RdSkew$. Similarly as in the proof of Proposition \ref{prop:ComplexCS} we can proceed by induction. Suppose $(v,M)\in \LC{d}i\cap\RdSkew$ and $f_{i+1}(v,M)\neq 0$. Then we have that

$$
\m_{1i}^2+\hdots+\m_{(i-1)i}^2\neq 0.
$$

By Proposition \ref{prop:Roti} we can find a rotation $A\in \NC{d}{d-i}$ such that $A\cdot (v,M)\in \LC{d}{i+1}$. Therefore if $(v,M)\in \UR$, there exists a rotation $A\in \OdR$ such that $A\cdot (v,M)\in \LL$.
\end{proof}

The following follows directly from Proposition \ref{cor:OFree}

\begin{corollary}\label{cor:OfreeR}
The action of $\OdR$ on $\UR \subset \RdSkew$ is free.
\end{corollary}

\begin{proposition}\label{prop:IDrealSep}
The set $\INV_d$ generates the invariant function field $\R(\RdSkew)^{\OdR}$ and separates orbits on $\UR$.
\end{proposition}

\begin{proof}
The fact that $\INV_d$ generates $\R(\RdSkew)^{\OdR}$ follows from Propositions \ref{prop:ComplexCS}, \ref{prop:RealCS} and Corollary \ref{cor:RealGenComplex}.
Using a similar argument as in Proposition \ref{prop:WInvGen}, we can see that $\INV_\Wd$ in \eqref{eq:WInv} separates orbits for the action of $\Wd$ on $\LLR\cap \UR$.
By Proposition \ref{prop:RealCS}, any orbit on $\UR$ meets $\LLR$, and the $\INV_d$ restrict to $\INV_\Wd$ on $\LLR$.
Thus $\INV_d$ is separating on $\UR$.
\end{proof}

We finish the section by constructing an explicit set of invariant polynomial functions that generate $\C(\CdSkew)^{\OdC}$. Consider the map \nomenclature[phik]{$\phi_k$}{The map $\phi_k \colon \CdSkew \rightarrow \C^d,\,(v,M) \mapsto M^kv$}
\begin{align*}
  \phi_k \colon \CdSkew &\rightarrow \C^d \\
  (v,M) &\mapsto M^kv.
\end{align*}
Then for the action of $A\cdot (v,M)$ we have that

$$
\phi_k\big(A\cdot (v,M)\big) = \phi_k\left( (Av, AMA^T) \right)=(AMA^T)^kAv=AM^kv=A\,\phi_k\big((v,M)\big).
$$

Thus the polynomial obtained by the dot-product of $\phi_k$ with itself is an invariant function on $ \CdSkew$ under $\OdC$.
We will show that the set of polynomial invariants
(defining $\DOT{a}{b} := \sum_i a_i b_i$)

\nomenclature[IM]{$\IM$}{The set of polynomial invariants generating $\C(\CdSkew)^{\OdC}$ given by $\DOT{\phi_k}{\phi_k}$, $1\leq k<d$}
\begin{equation}\label{eq:IM}
  \IM = \left\{ \DOT{v}{v},\, \DOT{M^kv}{M^kv}, |\, 1\leq k < d \right\}
\end{equation}
generate the field $\C(\CdSkew)^{\OdC}$ by restricting them to $\LL$.

\begin{lemma}\label{lem:vMform}
Consider a matrix $M$ of the form as in \eqref{eq:vMform}, i.e.\ such that $(v,M)\in \LL$.
Then for $1\leq k< d$, $M^k$ satisfies

\begin{itemize}
\item[(a)] $\displaystyle M^k(d-k,d) = \prod_{i=1}^k \m_{(d-i)(d-i+1)}$,
\item[(b)]$M^k(i,d) = 0$ for $i < d-k$,
\item[(c)]$M^k(i,d) \in \Q[\m_{(d-j)(d-j+1)}\, |\, 1\leq j < k]$  for $i > d-k$.
\end{itemize}
\end{lemma}

\begin{proof}
We proceed by induction. For $k=1$, $M^1=M$. Then $M^1$ satisfies (a)-(c), since $M(d-1,d)=\m_{(d-1)(d)}$ and $M(i,d)=0$ for $i<d-1$. Now suppose that (a)-(c) hold for $M^{k-1}$. We have that for $M^k=MM^{k-1}$,

\begin{align*}
M^k(1,d) &= \m_{12}M^{k-1}(2,d)\\
M^k(i,d) &= -\m_{i-1,i}M^{k-1}(i-1,d)+\m_{i,i+1}M^{k-1}(i+1,d)\\
M^k(d,d) &= -\m_{(d-1)d}M^{k-1}(d-1,d),
\end{align*}
where $1<i<d-1$. Note that $M^k(i,d)$ is linear combination of $M^{k-1}(i-1,d)$ and $M^{k-1}(i+1,d)$. By the induction hypothesis we know that $M^{k-1}(i,d)=0$ if $i < d-k+1$, and hence $M^k (i,d)=0$ when $i+1 < d-k+1$, or equivalently when $i<d-k$. This proves (b).

Suppose that $i > d-k$. Then $M^k(i,d)$ is linear in the terms

$$
\m_{i-1,i}, \quad \m_{i,i+1}, \quad M^{k-1}(i-1,d), \quad M^{k-1}(i+1,d),
$$
where $\m_{i-1,i}$ and $\m_{i,i+1}$ are of the form $\m_{(d-j)(d-j+1)}$ for $1\leq j< k$. By the induction hypothesis, the latter two terms are polynomials in $\m_{(d-j)(d-j+1)}$ where $1 \leq j < k-1$, proving (c).

Finally suppose that $i=d-k$. We have that

$$
M^k(d-k,d)= -\m_{d-k-1,d-k}M^{k-1}(d-k-1,d)+\m_{d-k,d-k+1}M^{k-1}(d-k+1,d).
$$

By the induction hypothesis we know that

$$
M^{k-1}(d-k+1,d)=\prod_{i=1}^{k-1} \m_{(d-i)(d-i+1)}\quad \text{and} \quad M^{k-1}(d-k-1,d)=0,
$$
which proves (a).
\end{proof}

\begin{lemma}\label{lem:IMgen}
  The polynomials obtained from restricting the functions in $\IM$ to $\LL$ generate the invariant rational function field $\C(\LL)^\Wd$.
\end{lemma}

\begin{proof}
First note that to restrict the polynomials in $\IM$ to $\LL$, we can assume that $(v,M)$ are of the form in 
\eqref{eq:vMform} and then compute the inner product. Then we can easily see that 
$$
\DOT{v}{v}|_{\LL} = \vv_d^2\quad \text{and}\quad \DOT{Mv}{Mv}|_{\LL} = \vv_d^2\m_{(d-1)d}^2.
$$
This implies that $\vv_d^2$ and $\m_{(d-1)d}^2$ are rational functions of $\DOT{v}{v}|_{\LL}$ and $\DOT{ Mv}{Mv}|_{\LL}$. We proceed by induction on $i$: suppose that $\m_{(d-i)(d-i+1)}^2$ is a rational function of $\DOT{ v}{v}|_\LL, \DOT{ Mv}{Mv}|_\LL,\hdots,\DOT{ M^iv}{M^iv}|_\LL$ for all $1\leq i < k$. By Lemma \ref{lem:vMform}, we know that
$$
\DOT{M^kv}{M^kv}|_\LL = \vv_d^2\prod_{i=1}^k \m_{(d-i)(d-i+1)}^2 + \vv_d^2 I\left(\m_{(d-1)d},\m_{(d-2)(d-1)},\hdots,\m_{(d-k+1)(d-k+2)}\right).
$$

Since $\DOT{ M^k}{ M^kv}|_\LL$ is an invariant function, as well as $\vv_d^2$ and $\m_{(d-i)(d-i+1)}^2$ for $1\leq i < d$, the function $I$ lies in $\C(\Wd)^\LL$. By the induction hypothesis and Proposition \ref{prop:WInvGen}, $I$ is a rational function of
$$
\DOT{ v}{v}|_\LL,\, \DOT{ Mv}{Mv}|_\LL,\,\hdots\, ,\DOT{ M^{k-1}v}{M^{k-1}v}|_\LL.
$$
Thus we can rewrite the above equality to
$$
\frac{\DOT{ M^kv}{M^kv} - \vv_d^2I\left(\DOT{ v}{v}|_\LL, \DOT{ Mv}{Mv}|_\LL,\hdots,\DOT{ M^{k-1}v}{M^{k-1}v}|_\LL\right)}{\displaystyle \vv_d^2\prod_{i=1}^{k-1} \m_{(d-i)(d-i+1)}^2} = \m_{(d-k)(d-k+1)}^2.
$$
By the induction hypothesis each $\m_{(d-i)(d-i+1)}^2$ for $1\leq i<k$ is a rational function of
$$
\DOT{ v}{v}|_\LL,\, \DOT{ Mv}{Mv}|_\LL,\,\hdots\, ,\DOT{ M^{k-1}v}{M^{k-1}v}|_\LL.
$$
This implies that $\m_{(d-k)(d-k+1)}^2$ is a rational function of
$$
\DOT{ v}{v}|_\LL,\, \DOT{ Mv}{Mv}|_\LL,\,\hdots\, ,\DOT{ M^{k}v}{M^{k}v}|_\LL.
$$
Therefore each element of $\INV_\Wd$ can be written as a rational function of polynomials in $\IM$ restricted to $\LL$. By Proposition \ref{prop:WInvGen}, $\IM$ restricted to $\LL$ is a generating set for $\C(\LL)^\Wd$. 
\end{proof}

\begin{proposition}
  \label{prop:invariantFields}
The set of polynomial invariants $\IM$ in $\eqref{eq:IM}$ generates both $\C(\CdSkew)^{\OdC}$ and $\R(\RdSkew)^{\OdR}$ and also separates orbits on $\CdSkew$ and $\RdSkew$.
\end{proposition}

\begin{proof}
By Proposition \ref{prop:ComplexCS} $\LL$ is a relative $\Wd$-section for the action of $\OdC$ on $\CdSkew$, and by Proposition \ref{prop:WInvGen} $\INV_\Wd$ is a generating set for $\C(\LL)^\Wd$. Thus by Lemma \ref{lem:IMgen} and Corollary \ref{cor:RelSecGen}, $\IM$ generates $\C(\CdSkew)^{\OdC}$. By Corollary \ref{cor:RealGenComplex} $\IM$ generates $\R(\RdSkew)^{\OdR}$.

By Proposition \ref{prop:GenSep} $\IM$ separates orbits on $\CdSkew$. By Proposition \ref{prop:IDrealSep} there exists a separating set of invariants in $\R(\RdSkew)^{\OdR}$, and hence $\R(\RdSkew)^{\OdR}$ separates orbits. Therefore, by Proposition \ref{prop:SepGenReal}, $\IM$ separates orbits on $\RdSkew$.
\end{proof}

\begin{remark}
 As a consequence, we in particular have that all $\DOT{M^k v}{M^kv}$ for $k\geq d$ can be expressed as rational functions with variables in $\INV_M$.
\end{remark}

\begin{example}\label{ex:generatorsO3field}
  By \Cref{prop:invariantFields} the field of invariants $\R(\R^3 \times \Skew3{\R})^{\O_3(\R)}$
  is generated by
  \begin{align*}
  \DOT{ v}{ v } &= \vv_1^2 + \vv_2^2 + \vv_3^2 \\
\DOT{ Mv}{ Mv } &= (\m_{12} \vv_1 - \m_{23} v_3)^2 + (\m_{13} \vv_1 + \m_{23} \vv_2)^2 + (\m_{12} v_2 + \m_{13} \vv_3)^2 \\
    \DOT{ M^2v}{ M^2v } &= 
    {\left({\m_{12}} {\m_{23}} {\vv_{1}} - {\m_{12}} {\m_{13}} {\vv_{2}} - {\left({\m_{13}}^{2} + {\m_{23}}^{2}\right)} {\vv_{3}}\right)}^{2} \\
    &\qquad+ {\left({\m_{13}} {\m_{23}} {\vv_{1}} + {\m_{12}} {\m_{13}} {\vv_{3}} + {\left({\m_{12}}^{2} + {\m_{23}}^{2}\right)} {\vv_{2}}\right)}^{2}\\
                               &\qquad + {\left({\m_{13}} {\m_{23}} {\vv_{2}} - {\m_{12}} {\m_{23}} {\vv_{3}} + {\left({\m_{12}}^{2} + {\m_{13}}^{2}\right)} {\vv_{1}}\right)}^{2}.
  \end{align*}
\end{example}

\section{\texorpdfstring{$\OdR$}{O(d, R)}-invariant iterated-integral signature}\label{sec:MFsection}

\subsection{Moving frame on \texorpdfstring{$\lie{n}2$}{g2((Rd))}.}\label{sec:MF2}

In this section, we construct a moving frame map for the action of $\O_2(\R)$ on $\lie{n}2$, and
show how this can be used to construct $\O_2(\R)$-invariants in $\lie{n}2$ and hence in the coefficients of the iterated-integralas signature of a curve $\CURVE$.

First consider the action on $\lie{2}2 = \R^2\oplus [\R^2,\R^2]$. We can denote any element of $\lie{2}2$ as $\cctwo$ with coordinates $c_{1}, c_{2}, $ and $c_{12}$. Through the isomorphism in \eqref{eq:lie2isom} we can consider $\cctwo$ as an element of $\R^2\oplus\mathfrak{so}(2,\R)$,

\begin{align*}
  \cctwo = (v,M) = 
  \left(\begin{bmatrix}
      c_{1} \\
      c_{2}
  \end{bmatrix},
  \begin{bmatrix}
    0 & c_{12} \\
    -c_{12} & 0
  \end{bmatrix}\right),
\end{align*}
and with action as in \eqref{eq:simpleForm}.
We will now show that $\O_2(\R)$ is free on $\lie{2}2$ and the following submanifold
$$
\K_{2,\le 2} := \left\{\cctwo \in \lie{2}2\, |\, c_{1} = 0; c_{2}, c_{12} >0\right\}
$$
\nomenclature[K22]{$\K_{2,\le 2}$}{The cross-section for the action of $\O_2(\R)$ on $\mathcal U_{2;\le 2}$}
is a cross-section for the action. Similarly to \Cref{ex:SO2mf}, we start by defining the group element

\begin{equation*}
  A(\cctwo)
  :=
  \frac{1}{\sqrt{c_{1}^2+c_{2}^2}}\begin{bmatrix}
    c_{2} & -c_{1} \\
    c_{1} &  c_{2}
  \end{bmatrix},
\end{equation*}
which is defined outside of $\{c_{1}=c_{2}=0\}$. For any such element $\cctwo \in \lie{2}2$, we have that
$$
A(\cctwo) \cdot \cctwo = \left(\begin{bmatrix}
    0 \\
    \sqrt{c_{1}^2+c_{2}^2}
  \end{bmatrix},
  \begin{bmatrix}
    0 & c_{12} \\
    -c_{12} & 0
  \end{bmatrix}\right).
$$

Unlike in \Cref{ex:SO2mf}, the action is not free on $\R^2$, the submanifold defined by $c_{1}=0,
c_{2} >0$ is not a cross-section, and $A(\cctwo)$ does not define a moving frame map. This is due to the
fact that a reflection about the $y$-axis will fix $v$, but change the sign of $M$. Thus to define a
moving frame map we must consider the diagonal action of $\O_2(\R)$ on all of $\lie{2}2$, not just
the action on $\lie{1}2=\R^2$. The map $\rhotwo: \mathcal U_{2;\le 2} \rightarrow \O_2(\R)$ given by
$$
\rhotwo(\cctwo) = \frac{1}{\sqrt{c_{1}^2+c_{2}^2}}\begin{bmatrix}
  \text{sgn}(c_{12})c_{2} & -\text{sgn}(c_{12})c_{1}\\
  c_{1} & c_{2}
  \end{bmatrix}
$$
\nomenclature[rho2]{$\rhotwo$}{The moving frame map $\rhotwo:\,\mathcal U_{2;\le 2} \rightarrow \O_2(\R)$ for the action of $\O_2(\R)$ on $\mathcal U_{2;\le 2}$}
defines the group element $\rhotwo(\cctwo)$ such that $\rhotwo(\cctwo)\cdot \cctwo \in \mathcal{K}$ where
$$
\mathcal U_{2;\le 2} = \left\{ \cctwo=(v,M) \,|\, v\neq \begin{bmatrix} 0\\0\end{bmatrix}, c_{12}\neq 0\right\} \subset \lie{2}2.
$$
\nomenclature[U22]{$\mathcal U_{2;\le 2}$}{The domain of the moving frame $\rhotwo$, a Zariski-open subset of $\lie{2}2$}

Note that \(\K_{2;\le 2}\) 
is a subset of $L^{(1);\R}_2$
and 
and \(\mathcal U_{2;\le 2}\) is equal to \(\UR\), both defined in \Cref{prop:RealCS}.
The (unique) intersection point of the orbit $\O_2(\R)\cdot \cctwo$ with $\K_{2; \le 2}$ is given by \(\rhotwo(\cctwo)\cdot\cctwo\).
Since the action is free on $\lie{2}2$ (\Cref{cor:OfreeR}), the map $\rhotwo$ defines a moving frame with cross-section $\mathcal{K}$. This immediately implies that the coordinates of $\rhotwo(\cctwo)\cdot \cctwo$ are invariants for the action of $\O_2(\R)$ on $\lie{2}2$\footnote{The constant functions are referred to as the \textit{phantom invariants}.}:
$$
\sqrt{c_{1}^2+c_{2}^2}, \quad |c_{12}|.
$$
Furthermore any two elements $\cctwo, \tilde{\cc}_{\leq 2} \in \lie{2}2$ are related by an element of $\O_2(\R)$ if and only if
$$
\sqrt{c_{1}^2+c_{2}^2} = \sqrt{\tilde{c}_{1}^2+\tilde{c}_{2}^2} \quad \text{and}\quad |c_{12}|=|\tilde{c}_{12}|.
$$

For any path $\CURVE$ in $\R^2$, let $\cctwo(\CURVE)$ denote the element of $\lie{2}2$ given by $\proj_{\le 2}(\log(\IIS(\CURVE)))$. Then we can define the ``invariantized'' path $Y:=\rhotwo(\cctwo(\CURVE)) \cdot \CURVE$. The above statement implies that for any two paths $\CURVE, {\CURVE}'$, we have that $\cctwo(Y) =
\cctwo({Y}')$ if and only if there exists some $g\in\O_2(\R)$ such that 
\[ g\cdot\cctwo(\CURVE)=\cctwo(g\cdot \CURVE)=\cctwo({\CURVE}').\]
In particular, since the $\log$ map is an equivariant bijection, the same holds true for the $\IIS$ of a path under the projection $\proj_{\leq 2}$.

Given a path $\CURVE$ starting at the origin, the values of $c_{1}(\CURVE), c_{2}(\CURVE)$ correspond to $x$ and $y$ values of $\CURVE(1)$.
Similarly the value of $c_{12}(\CURVE)$ corresponds to the so-called Lévy area traced by \(\CURVE\)
(see \cite[Section 3.2]{DR2019} in the context of classical invariant theory).
Thus the moving frame map applied to such a path $\CURVE$, rotates the end point $\CURVE(1)$ to the $y$-axis (and reflects about the $y$-axis if the Lévy area is negative). 

The resulting invariants on $\lie{2}2$ are perhaps unsurprising, but the above method also yields
$\O_2(\R)$-invariants on $\lie{n}2$ for an arbitrary truncation order $n$, as we now show.

We define a map $\rhotwon:\mathcal U_{2;\le n}\subset\lie{n}2 \rightarrow \O_2(\R)$\nomenclature[rho2]{$\rhotwon$}{The moving frame map for the action of $\O_2(\R)$ on $\mathcal U_{2;\le n}$, where $\rhotwon(\ccn)=\rhotwo(\proj_{\le 2}\ccn)$} by
$$
\rhotwon(\ccn) = \frac{1}{\sqrt{c_{1}^2+c_{2}^2}}\begin{bmatrix}
  \text{sgn}(c_{12})c_{2} & - \text{sgn}(c_{12})c_{1}\\
  c_{1} & c_{2}
  \end{bmatrix}
$$
for any $\ccn\in {\mathcal U}_{2;\le n}$ where 
$$
{\mathcal U}_{2;\le n} := \proj_{\le n \to \le 2}^{-1}\left( \mathcal U_{2;\le 2} \right) \subset \lie{n}2,
$$
with $\proj_{\le n\to \le 2}$ denoting the canonical projection from $\lie{n}2$ onto $\lie{2}2$.

Since $\O_2(\R)$ acts
diagonally on the whole of \(\lie{n}2\), $\rhotwon$ is a moving frame map on $\lie{n}2$ with cross-section ${\mathcal{K}}_{2;\le n}$ where 
$$
{\mathcal{K}}_{2;\le n} := \proj_{\le n \to \le 2}^{-1}\left( \K_{2;\le 2} \right) \subset \lie{n}2.
$$

Then the resulting coordinate functions of $\rhotwon(\mathbf \ccn)\cdot\ccn\in\lie{n}2$ are $\O_2(\R)$ invariants for the action on $\lie{n}2$ (see \Cref{sec:inv.planar} for a
more detailed investigation of these invariants), and hence $\O_2(\R)$ invariants for paths in $\R^2$. Furthermore, for any truncation order $n$ and paths $\CURVE,{\CURVE}'\in\R^2$, we have that
$\ccn(Y) = \ccn({Y}')$ if and only if there exists some element
of $\O_2(\R)$ such that $g\cdot \ccn(\CURVE) =\ccn({\CURVE}').$ The
following is then true by induction and the fact that the $\log$ map is an equivariant bijection.

\begin{proposition}
  Let $\CURVE, {\CURVE}'$ be paths in $\R^2$ such that
  \begin{align*}
  \cctwo(\CURVE):= \proj_{ \le 2}(\log(\IIS(\CURVE))),\qquad \cctwo({\CURVE}'):= \proj_{ \le 2}(\log(\IIS({\CURVE}'))),
  \end{align*}
  are elements of $\mathcal U_{2;\le 2}$.
  Define 
  \begin{align*}
    Y\coloneq \rhotwo(\cctwo(\CURVE))\cdot \CURVE,\qquad Y'\coloneq\rhotwo(\cctwo({\CURVE}'))\cdot \CURVE'.
  \end{align*}
	Then $\IIS(Y) = \IIS({Y}')$ if and only if there exists $g\in \O_2(\R)$ such that $\IIS(g\cdot \CURVE) = \IIS({\CURVE}')$.
\end{proposition}

Therefore two paths, starting at the origin are equivalent up to tree-like extensions and action of $\O_2(\R)$ if and only if $\IIS(Y) = \IIS({Y}')$. In this sense, the moving frame map $\rhotwo$ yields a method to invariantize a path $\CURVE$. In the following section, we show that this construction extends to paths in $\R^d$.

\subsection{Moving Frame on \texorpdfstring{$\lie{n}d$}{g2((Rd))}}\label{sec:MF}%
As for $\O_2(\R)$ on $\R^2$, the action of $\OdR$ on
paths in $\R^d$ induces an action on its (truncated) signature that coincides with the diagonal action on the ambient space $T_{\le n}(\R^d)$.
The induced action on the log-signature coincides with this diagonal action as well,
when considering $\lie{n}d$ as a subspace of $T_{\le n}(\R^d)$.

\newcommand\CROSS{{\K_{d;\le n}}}
Let $\ccn$\nomenclature[cn]{$\ccn$}{An element of $\lie{n}d$ with coordinates given by $c_{i_1i_2\dotsm i_m}$ for $m\leq n$} be an element of $\lie{n}d$ with coordinates given by $c_{i_1i_2\dotsm i_m}$ for $m\leq n$. We define the following submanifold of $\lie{n}d$:
\begin{equation}\label{eq:LAxsection}
  \CROSS = \left\{ c_{i}=0, c_{j(i+1)}=0, c_{d}>0,\ c_{i(i+1)}>0\,|\, 1\leq i\leq d-1, 1\leq j <i \right\} \subset \lie{n}d
\end{equation}
\nomenclature[Kd<=n]{$\CROSS$}{The cross-section for the action of $\OdR$ on $\UU$}
Let $\proj_{\le 2} : \lie{n}d \rightarrow \lie{2}d$ be the projection onto the first two levels (\Cref{ss:iis}). The projection of this submanifold onto $\lie{2}d$, $\proj_{\leq 2}(\CROSS)$,
is equal (up to the identification $\lie{2}d \cong ..$) to the real, positive points of $\LL$ in \eqref{eq:vMform} where 
$$
  \left(
  \begin{bmatrix}
    c_{1} \\
    \vdots \\
    c_{d}
  \end{bmatrix},
  \begin{bmatrix}
    0 & c_{12} & \dots &c_{1{d}} \\    
    -  c_{12}& 0 & \dots &c_{2d} \\    
    \dots& \dots & \ddots &\dots \\    
    -  c_{1d}& -c_{2d} & \dots &0
  \end{bmatrix}\right) = (v,M).
$$
Similarly we can define the analogue to $\UC$ in \eqref{eq:UC}. Consider the rational functions on $\lie{n}d$ given by
$$
F_i(\ccn):=f_i(v,M)|_{v_j=c_{j} \, m_{k\ell}=c_{k\ell}}
$$
for $1\leq i\leq d$ where $f_i(v,M)$ is given in Proposition \ref{prop:ComplexCS}. By Proposition \ref{prop:RealCS}, the functions $F_i$ are rational functions on $\lie2d$ with real coefficients. Then the following is a Zariski-open subset of $\lie{n}d$,
\newcommand\UU{\mathcal U_{d;\le n}}
\begin{align*}
\UU
&:= \left\{\ccn\in \lie{n}d\,\,|\,\, F_i(\ccn)\neq 0, \forall i, 1\leq i\leq d\right\},\\
\end{align*}\
\nomenclature[Ud<=n]{$\UU$}{The domain of the moving frame $\rhodn$, a Zariski-open subset of $\lie{n}d$}
where $\proj_{\leq 2}(\UU)= \UC$ if we identify $\cctwo$ with $(v,M)$ as above. In particular, both $\UU$ and $\CROSS$ are completely characterized by $\proj_{\leq 2}(\ccn)$, i.e.
\begin{align*}
\UU &= \proj_{\le n \to \le 2}^{-1}\left( \proj_{\leq 2} (\UU) \right) \subset \lie{n}d\\
\CROSS&= \proj_{\le n \to \le 2}^{-1}\left(  \proj_{\leq 2} (\CROSS) \right) \subset \lie{n}d,
\end{align*}
with $\proj_{\le n\to \le 2}$\nomenclature[projn2-1]{$\proj_{\le n\to \le 2}$}{The canonical projection $\proj_{\le n\to \le 2}:\,\lie{n}d\to\lie{2}d$} denoting the canonical projection from $\lie{n}d$ onto $\lie{2}d$.

We now show that on the subset $\UU \subset \lie{n}d$ the submanifold $\CROSS$ is a cross-section, which induces a moving frame.
\begin{lemma}\label{lem:N2transversality}
For any point $\cctwo \in\K_{d;\le 2} \cap U_{d;\le 2}$, the orbit $\OdR \cdot \cctwo$ and $\K_{d; \le 2}$ intersect transversally.
\end{lemma}

\begin{proof}
	First, we recall that, by definition, \(\OdR\cdot\cctwo\) and \(\K_{d;\le 2}\) intersect transversally
	if and only if, at
	every point \(q\) in the intersection, the tangent spaces \(T_q(\OdR\cdot\cctwo)\) and \(T_q\K_{d;\le 2}\)
	generate the whole ambient space \(\lie 2d\),
	that is 
	\begin{equation*}
	   T_q(\OdR\cdot\cctwo)+T_q\K_{d;\le 2}=\lie2d.
	\end{equation*}

	Now, at a point \(q=A\cdot\cctwo=(Av,AMA^{\top})\) in the orbit, the tangent space has the form
	\begin{equation}
		T_q\Big(\OdR\cdot\cctwo\Big)=\{(HAv,[H,AMA^{\top}]):H\in\Skew d\R\}.
		\label{eq:TqOdR}
	\end{equation}
	Indeed, recall that for a manifold \(M\), its tangent space at a point \(q\) is the linear space \(T_qM\coloneqq\{\gamma'(0):\gamma\text{ curve s.t. }\gamma(0) = q\}\). A curve \(\gamma\) on \(\OdR\cdot\cctwo\) such that \(\gamma(0)=q\) has the form \(\gamma(t)=(L(t)A)\cdot\cctwo\) for some curve \(t\mapsto L(t)\) in \(\OdR\) such that \(L(0)=I\).
	Hence,
	\begin{align*}
	  \gamma'(0)&= (L'(0)Av, L'(0)AMA^\top+ AMA^\top L'(0)^\top)\\
	  &=(L'(0)Av, L'(0)AMA^\top-AMA^\top L'(0)).
  \end{align*}
	
	The tangent space to the cross section is
	$$
		T_q\K_{d;\le 2}=\{c_{i}=0,c_{j(i+1)}=0 : 1\le i\le d-1, 1\le j<i\}.
	$$
	We note that
	\[
		\dim T_q\K_{d;\le 2}=d,\quad\dim T_q\Big(\OdR\cdot\cctwo\Big)=\frac{d(d-1)}{2},
	\]
	where the second equality since the action of $\OdR$ is free on $ U_{d;\le 2}$ by Corollary \ref{cor:OfreeR}.
	Thus we have that \(\dim T_q\K_{d;\le 2}+\dim T_q\Big(\OdR\cdot\cctwo\Big)=\dim\lie 2d\).
	Therefore, we only need to show that \(T_q\K_{d;\le 2}\cap T_q\Big(\OdR\cdot\cctwo\Big)=\{0\}\) for all
	\(q\in\K_{d;\le 2}\cap\left(\OdR\cdot\cctwo\right)\).
	
	Let \((\Gamma_{i,j}:1\le i<j\le d)\) be the standard basis of \(\Skew d\R\), that is,
	\((\Gamma_{i,j})_{k,l}=\delta_{i,k}\delta_{j,l}-\delta_{j,k}\delta_{i,l}\).
	It is not hard to show that the commutation relations
	\begin{equation}
		[\Gamma_{i,j}, \Gamma_{k,k + 1}] = \Gamma_{k + 1,j} \delta_{i,k} +
		\Gamma_{i,k + 1} \delta_{j,k} - \Gamma_{k,j} \delta_{i, k + 1} - \Gamma_{i, k} \delta_{j, k + 1}
		\label{eq:Gcomm}
	\end{equation}
	hold for all \(1\le k<d\) and \(1\le i<j\le d\).
	By \cref{eq:TqOdR}, a generic element \(p\in T_q\Big(\OdR\cdot\cctwo\Big)\) has the form
	\(p=(HAv,[H,AMA^{\top}])\) with
	\[
		H=\sum_{1\le i<j\le d}h_{i,j}\Gamma_{i,j}\in\Skew d\R.
	\]
	But since \(q=(Av,AMA^{\top})\in\K_{d;\le 2}\),
	\[
		Av=\alpha e_d,\quad AMA^\top=\sum_{k=1}^{d-1}\beta_k\Gamma_{k,k+1}
	\]
	with \(\alpha>0\), and \(\beta_k>0\) for all \(k\in\{1,\dotsc,d-1\}\).
	If \(p\) also belongs to \(T_q\K_{d;\le 2}\), then we have in particular that
	\[
		HAv=\sum_{i=1}^{d-1}h_{i,d}e_i=\alpha'e_d,
	\]
	for some \(\alpha'\in\R\), thus \(h_{i,d}=0\) for all \(i\in\{1,\dotsc,d-1\}\).
	Now we show that \(h_{i,j}=0\) for all \(1\le i<j\le d-1\) by induction on \(r\coloneq d-1-j\).
	By \cref{eq:Gcomm}, we see that
	\[
		[H,AMA^\top]=\sum_{1\le i<j\le d-1}\sum_{k=1}^{d-1}h_{i,j}\beta_k(\Gamma_{(k + 1),j} \delta_{i,k} +
		\Gamma_{i,(k + 1)} \delta_{j,k} - \Gamma_{k,j} \delta_{i, (k + 1)} - \Gamma_{i, k} \delta_{j, (k
	+ 1)}),
	\]
	so that
	\[
		[H,AMA^\top]_{i,d-1}=h_{i,d-1}\beta_{d-1}=0.
	\]
	for \(i\in\{1,\dotsc,d-2\}\). Therefore, \(h_{i,d-1}=0\) for all \(i\in\{1,\dotsc,d-2\}\), and the claim is proven when \(r=0\).
	Suppose it is true for all \(r'<r\). Then
	\[
		[H,AMA^\top]_{i,d-1-r}=h_{i,d-1-r}\beta_{d-1-r}=0
	\]
	for \(i\in\{1,\dotsc,d-2-r\}\), hence \(h_{i,d-1-r}=0\) for all \(i\in\{1,\dotsc,d-2-r\}\).
	Finally, we have \(H=0\) thus \(p=(HAv,[H,AMA^\top])=0\).

	We have shown that if \(q\in\OdR\cdot\cctwo\cap\K_{d;\le 2}\) then \(\dim T_q\Big(\OdR\cdot\cctwo\Big)+\dim
	T_q\K_{d;\le 2}=\dim\lie 2d\) and \(T_q\Big(\OdR\cdot\cctwo\Big)\cap T_q\K_{d;\le 2}\) is trivial.
	It follows that if \(q\in\left(\OdR\cdot\cctwo\right)\cap\K_{d;\le 2}\), then
	\[
		\lie 2d=T_q\Big(\OdR\cdot\cctwo\Big)\oplus T_q\K_{d;\le 2},
	\]
	and in particular \(\OdR\cdot\cctwo\) and \(\K_{d;\le 2}\) intersect transversally.
\end{proof}

\begin{theorem}\label{thm:MainThm}
  The submanifold $\CROSS$ in \eqref{eq:LAxsection} is a cross-section for the action of $\OdR$ on $\UU \subset \lie{n}d$.
  In particular $\CROSS$ induces a moving frame map $\rhodn : \UU \rightarrow \OdR$.\nomenclature[rhod]{$\rhodn$}{The moving frame map for the action of $\O_d(\R)$ on $\UU$, where $\rhodn(\ccn)=\rhod(\proj_{\le 2}\ccn)$}
\end{theorem}
\begin{proof}
We first claim that $\CROSS$ intersects each orbit in $\UU$ at a unique point.
Denote the linear span of $\CROSS$ as
$$
K := \left\{ c_{i}=0, c_{j(i+1)}=0\,|\, 1\leq i\leq d-1, 1\leq j <i \right\} \subset \lie{n}d.
$$
Note that the action on $\proj_{\le 2} \lie{n}d = \lie2d$ is isomorphic to the action on $\RdSkew$ given in \eqref{eq:simpleForm}.
Thus for any $\ccn \in \UU$, by Proposition \ref{prop:RealCS} and the diagonality of the action (see \Cref{ss:invariants}),
there exists an element of $g\in \OdR$ such that $g\cdot \ccn = \tilde{\cc}_{\leq n} \in K$.

Consider the subgroup $W_\R \subset \OdR$ of diagonal matrices $w$ with diagonal entries $w_{jj}\in \{-1,1\}$, $1\leq j\leq d$. By Proposition \ref{prop:LiNi} any element of $W_\R$ sends a point in $K$ to $K$.
For any $\tilde{\cc}_{\leq n} \in K$, the action of $W_\R$ on the coordinates $\proj_{\leq
2}(\ccn)=\cctwo$ is given by the following (see \eqref{eq:Waction}):
\begin{align*}
  c_{d} \mapsto w_{dd} c_{d}, \quad c_{i(i+1)} \mapsto w_{ii}w_{(i+1)(i+1)}c_{i(i+1)}.
\end{align*}

The element $w\in W_\R$ such that $w_{jj}= -1$ for $1\leq j \leq d$ changes only the sign on $c_{d}$.
The element $w\in W_\R$ where $w_{jj} = -1$ for $1\leq j \leq i$ and $w_{jj} = 1$ for $i < j \leq d$
changes only the sign of $c_{i(i+1)}$.
Thus there exists $g\in W_\R$ such that $g\cdot \tilde{\cc}_{\leq n} \in \CROSS$, implying $\CROSS$ intersects each orbit in $\UU$.

Now suppose that for some $\ccn\in \CROSS$,
$g\in \OdR$ we have $g\cdot \ccn\in \CROSS$.
We show that this implies $g = \id$.
Since the action of $\OdR$ on $T_1(\R^d)$ is isomorphic to
the canonical action on $\R^d$,
$g\in \O^{d-1}_d(\R)$ (recall the notation after \eqref{eq:SOimatrix}).
By Proposition \ref{prop:Roti},
the action of $\O^{d-1}_d(\R)$ on the coordinates $c_{1d}, c_{2,d},\hdots, c_{(d-1)d}$ of $\ccn$
is isomorphic to the canonical action on $\R^{d-1}$. Thus we deduce that $g$ must be in $\O^{d-2}_d(\R)$.
Iterating, we obtain that $g$ must be the identity, as claimed, implying that $\CROSS$ intersects each orbit in $\UU$ exactly once.

We now show that the intersection with each orbit is transverse. By Corollary \ref{cor:OfreeR} the action is free on $U_{d;\le 2}$, and thus on $\UU$. Since the action is free on $\UU$, each orbit $\OdR \cdot \ccn$ is smooth and of dimension $n(n-1)/2$ (see Proposition \ref{prop:alggroupfacts}).  Let $\ccn$ be a point in $\CROSS$. Since $\CROSS$ is on open subset of the linear space $K$, we have $T_{\ccn}\CROSS =K$.
Since $\CROSS$ and $\OdR \cdot \ccn$ are of complementary dimension, $\CROSS$
intersects $\OdR \cdot \ccn$ transversally if and only if the dimension of the span of their tangent
spaces is equal to the dimension of $\UU$.

Since $\OdR$ acts diagonally 
we have that
\begin{align*}
  \proj_{\le 2}(T_{\ccn}\CROSS + T_{\ccn} (\OdR \cdot \ccn)) = \proj_{\le 2}(T_{\ccn}\CROSS) + \proj_{\le 2}(T_{\ccn} (\OdR \cdot \ccn))&\\
                                               = T_{\proj_{\le 2}(\ccn)} \proj_{\le 2}(\CROSS) + T_{\proj_{\le 2}(\ccn)} \left(\OdR \cdot \proj_{\le 2}(\ccn)\right)&,
\end{align*}
where $V + W$ denotes the span of two subspaces $V,W$. Then by Lemma \ref{lem:N2transversality} 
\begin{align*}
  \proj_{\le 2}(T_{\ccn}\CROSS + T_{\ccn} (\OdR \cdot \ccn)) = \lie{2}d.
\end{align*}
Since for any vector $v \in T_{\proj_{\le 2}(\ccn)} \proj_{\le 2}(\CROSS)$,
$\langle v \rangle \oplus \lie{3}d$ is a subspace of $T_{\ccn}\CROSS$, we have that
$T_{\ccn}\CROSS + T_{\ccn} (\OdR \cdot \ccn)=\lie{n}d$. Thus $\CROSS$ and
$\OdR\cdot \ccn$ intersect transversally.

Therefore $\CROSS$ intersects transversally each orbit of $\UU$ at a unique point, and hence by definition is a cross-section for this action.
The free and algebraic action of $\OdR$ on $\UU$ satisfies the hypothesis of Theorem \ref{thm:mfxsection} (see Remark \ref{rem:AlgRegularity}), and hence there exists a moving frame map $\rhodn : \UU \rightarrow \OdR$ taking each element of $\UU$ to the unique intersection point of its orbit and $\CROSS$.
\end{proof}

The proof of Proposition \ref{prop:ComplexCS} provides a road-map for explicitly finding the element of $\OdR$ taking any point $\ccn\in \UU$ to $\CROSS$, and hence $\rhodn(\ccn)\cdot \ccn$. By successively applying rotations, one can bring $\ccn$ to the cross-section $\CROSS$. For an example of doing this in practice see Example \ref{ex:mfmomentcurve}.

An important consequence of Theorem \ref{thm:MainThm} is the following corollary.

\begin{corollary}
Two elements $\ccn,\tilde{\cc}_{\leq n} \in \UU$ lie in the same orbit if and only if they take the same value on the cross-section $\CROSS$, i.e. if and only if $\rhodn(\ccn) \cdot \ccn = \rhodn(\tilde{\cc}_\leq{n}) \cdot \tilde{\cc}_{\leq n}$.
\end{corollary}

Thus to find a unique representative of the orbit of $\ccn\in \UU$ we can ``invariantize'' $\ccn$ by computing $\rhodn(\ccn) \cdot \ccn$, and the smooth functions defining the non-zero coordinates of $\CROSS \cap \OdR \cdot \ccn$ are invariant functions which characterize the orbit. Note that the cross-section $\CROSS$ and the moving frame only depend on the $\lie{2}d$ coordinates. In particular we have that for any path $\CURVE$ such that $\ccn(\CURVE)=\proj_{\leq n}(\log(\IIS(\CURVE)))\in U^d_n$
$$
\rhodn(\ccn(\CURVE)) = \rhodn(\proj_{\leq 2}(\ccn(\CURVE))) =: \rhod(\cctwo(\CURVE))
$$
\nomenclature[rhod]{$\rhod$}{The moving frame map for the action of $\OdR$ on $\mathcal{U}_{d;\le 2}$}
which implies that the ``invariantization'' of a path $Y:= \rhod(\cctwo(\CURVE))\cdot \CURVE$ is well-defined. This is due to the diagonal nature of the action of $\OdR$ on $\lie{n}d$, and the fact that $\dim(\OdR) < \dim(\lie{2}d)$. Since the action of the coordinates on $\lie{2}d$ is not affected by the higher level coordinates, we can define a cross-section on $\lie{2}d$ that extends naturally to $\lie{n}d$. For higher-dimensional groups one may have to consider a cross-section on $\lie{3}d$ or higher.

As a consequence, the infinite log signature (and thus the iterated-integrals signature) of a path $\CURVE$ under the action of $\OdR$ is characterized by its value on the cross-section.

\begin{theorem}\label{thm:IISequiv}
  For any two paths $\CURVE, \tilde{\CURVE}$ in $\R^d$ such that $\cctwo(\CURVE) := \proj_{\le2}(\log(\IIS(\CURVE))), \cctwo(\tilde{\CURVE}) := \proj_{\le2}(\log(\IIS(\tilde{\CURVE})))$ are elements of 
  $\mathcal U_{d;\le 2}$, define
  \begin{align*}
  Y:=\rhod\left(\cctwo(\CURVE)\right) \CURVE,\qquad \tilde{Y}:=\rhod\left(\cctwo(\tilde{\CURVE})\right)\tilde{\CURVE}.
  \end{align*}
  Then $\IIS(Y) = \IIS(\tilde{Y})$ if and only if there exists $g\in \OdR$ such that $\IIS(g\cdot \CURVE) = \IIS(\tilde{\CURVE})$. 
\end{theorem}

\subsection{Towards a fundamental set of polynomial invariants}\label{sec:towardpoly}
The non-constant-zero coordinates of $\rhodn(\ccn)\cdot\ccn$ form a fundamental set of invariants for the action of $\O_d(\R)$ on $\UU$, since for any $\O_d(\R)$ invariant function $I:\,\UU\to\mathbb{R}$ we have $I(\ccn)=I(\rhodn(\ccn)\cdot\ccn)$. The coordinate functions of $\rhodn(\ccn)\cdot\ccn$ are, however, in general not rational. However polynomial invariants of the iterated-integral signature have a rich structure and are often desired (see \cite{DR2019}), and hence it is of strong interest to obtaining a minimal set of \emph{polynomial} invariants seperating orbits. 

In fact, there is even the following conjecture \cite[Conjecture~7.2]{DR2019} that polynomial invariants seperate orbits of paths up to tree like equivalence for any subgroup of $\SL_d(\R)$.
\begin{conjecture}\label{conj:diehlreizenstein}\textnormal{(Diehl-Reizenstein)}
Let $\CURVE,\CURVE':\,[0,T]\to\mathbb{R}^d$ be two curves such that
\begin{equation*}
    \langle \IIS(\CURVE),\varphi\rangle=\langle\ISS(\CURVE'),\varphi\rangle
\end{equation*}
for any $\varphi\in T(\R^d)$ such that $\tilde{\phi}_{A^\top}(\varphi)=\varphi$ for any $A\in G$. Then, there is $A\in G$ and a curve $\bar{\CURVE}$ which is tree-like equivalent to $\CURVE$ such that
\begin{equation*}
    A\bar\CURVE=\CURVE'.
\end{equation*}
\end{conjecture}

While a proof of this conjecture for any compact group $G$ is finished and part of work in progress by J.D., Terry Lyons, Hao Ni and R.P., we are here interested in a 'constructive' proof which leads to an algorithm for the computation of a minimal set of polynomial seperating orbits. The following definition and proposition now provide a sufficient condition for a moving frame to lead to a fundamental set of invariants consisting only of polynomial invariants.

\begin{definition}
 A moving frame $\varrho:\,U\to G$ for the action of $G$ on $U$, where $U$ is a non-empty, $G$-invariant, semialgebraic subset of $\lie{n}d$, is called \textbf{almost-polynomial}, 
if there are %
maps $\lambda:\,U\to\GL_d(\R)$ and $\kappa:\,\lie{n}d\to\GL_d(\R)$, where $\lambda$ is $G$-invariant, such that $\lambda(\ccn)$ is diagonal for all $\ccn\in U$, such that $\lambda_{ii}\varrho_{ij}\in\R[\lie{n}d]$ for all $i,j=1,\dots ,d$ and
\begin{equation}\label{eq:lambdakappa}
      \lambda(\ccn)=\kappa\big(\lambda(\ccn)\varrho(\ccn)\cdot\ccn\big)
\end{equation}
for all $\ccn\in U$.

\begin{remark}
 Equation \eqref{eq:lambdakappa} may seem odd as an asumption at first sight, however it is necessary for the coordinates of $\lambda(\ccn)\varrho(\ccn)\cdot\ccn$ to form a fundamental set of invariants: Since we assume $\lambda$ to be an invariant function it must functionally depend on any fundamental set of invariants.
\end{remark}

\begin{example}\label{ex:almost-polynomial}
Looking at the example $\rhotwon$, we introduce
\begin{equation*}
    \lambda_2(\ccn)=
    \begin{bmatrix}
        |c_{12}|\sqrt{c_{1}^2+c_{2}^2} & 0\\
        0 & \sqrt{c_{1}^2+c_{2}^2}
    \end{bmatrix}.
\end{equation*}
As we see in Section \ref{sec:inv.planar}, $\lambda_2$ is invariant under $\O_2(\R)$. Then,
\begin{equation*}
    \lambda_2(\ccn)\rhotwon(\ccn)=
    \begin{bmatrix}
        c_{12}c_{2} & -c_{12}c_1\\
        c_1 & c_2
    \end{bmatrix},
\end{equation*}
whose entries are polynomial functions. In particular, the non-zero coordinates of $\lambda_2(\ccn)\rho_2(\ccn)\cdot\ccn$ are given by
\begin{equation*}
    \hat c_{2}:=c_{1}^2+c_{2}^2,\quad \hat c_{12}:=c_{12}^2(c_{1}^2+c_{2}^2)
\end{equation*}
We finally obtain $\lambda_2(\ccn)=\kappa_2(\lambda_2(\ccn)\varrho(\ccn)\cdot\ccn)$ via
\begin{equation*}
    \kappa_2(\hat{\mathbf{c}}_{\le n})=
    \begin{bmatrix}
        \sqrt{\hat c_{12}} & 0\\
        0 & \hat c_{2}
    \end{bmatrix},
\end{equation*}
showing that $\rho_2$ is an almost-polynomial moving frame.
\end{example}

\end{definition}
\begin{proposition}\label{prop:almost-polynomial}
If $\varrho:\,U\to G$ is an almost-polynomial moving frame for the action of $G$ on $U$, 
then the non-zero components of $\lambda(\ccn)\varrho(\ccn)\cdot\ccn$ form a fundamental set of invariants consisting only of polynomial invariants.
\end{proposition}
\begin{proof}
 Obviously $\ccn\mapsto\lambda(\ccn)\varrho(\ccn)\cdot\ccn$ is a polynomial map on $U$ since $\lambda_{ii}\mu_{ij}$ is polynomial. The components of $\varrho(\ccn)\cdot\ccn$ form a fundamental set of invariants since $\varrho$ is a moving frame which implies that $I(\ccn)=I(\varrho(\ccn)\cdot\ccn)$ for any invariant function $I:\,U\to\mathbb{R}$. Since $\lambda$ is $G$ invariant by assumption, we have that the components of $\lambda(\ccn)\varrho(\ccn)\cdot\ccn$ are invariants. Furthermore
 \begin{equation*}
     \varrho(\ccn)\cdot\ccn=\kappa(\lambda(\ccn)\varrho(\ccn)\cdot\ccn)^{-1}\lambda(\ccn)\varrho(\ccn)\cdot\ccn,
 \end{equation*}
implies that the non-zero components of $\lambda(\ccn)\varrho(\ccn)\cdot\ccn$ form a fundamental set of invariants, too.
\end{proof}

We formulate the following conjecture for the moving frame defined in the previous section.

\begin{conjecture}\label{conj:almostpoly}
For any $d\geq 2$, the moving frame $\rhodn:\,\UU\to\OdR$ is almost-polynomial.
\end{conjecture}
We will see in the next section that this conjecture at least also holds true for $d=3$, however, the conjecture remains open for $d>3$. This in particular means that we have our 'constructive' proof for Conjecture \ref{conj:diehlreizenstein} restricted to paths $\CURVE$ such that $\ccn(\CURVE)\in\UU$ in the special case of $\O_d(\R)$ for $d=2,3$. We hope to extend this result to all dimensions, all paths and to further compact groups in future work.

\section{Invariants of planar and spatial curves}
\label{sec:inv.planar}

\subsection{Planar curves}
In Section \ref{sec:MF2} we detailed a moving frame construction for $\lie{n}2$ under $\O_2(\R)$ for any truncation order $n$. In particular on the subset
$$
\mathcal U_{2; \le n} = \left\{ \ccn\in \lie{n}2 \,|\, (c_{1},c_{2}) \neq (0,0), c_{12} \neq 0\right\},
$$
the map $\rhotwon : \mathcal U_{2;\le n} \rightarrow \O_2(\R)$ defined by $\rhotwon(\ccn)=\mu_2(\ccn)\nu_2(\ccn)$ for $\ccn\in \lie{n}2$, where
$$
\mu_2(\ccn) = \frac{1}{\sqrt{c_{1}^2+c_{2}^2}}
\begin{bmatrix}
    \text{sgn}(c_{12}) & 0\\
    0 & 1
\end{bmatrix}, \quad
\nu_2(\ccn)=
\begin{bmatrix}
  c_{2} & - c_{1}\\
  c_{1} & c_{2}
  \end{bmatrix},
$$
 is a moving frame map, bringing any element of $\lie{n}2$ to the intersection of its orbit with the cross-section (see \Cref{fig:area})
$$
\K_{2;\le n} = \left\{\ccn \in \lie{n}2\, |\, c_{1} = 0, c_{2}, c_{12} >0\right\}.
$$
\begin{figure}[!ht]
    \centering
    \includegraphics[width=0.6\textwidth]{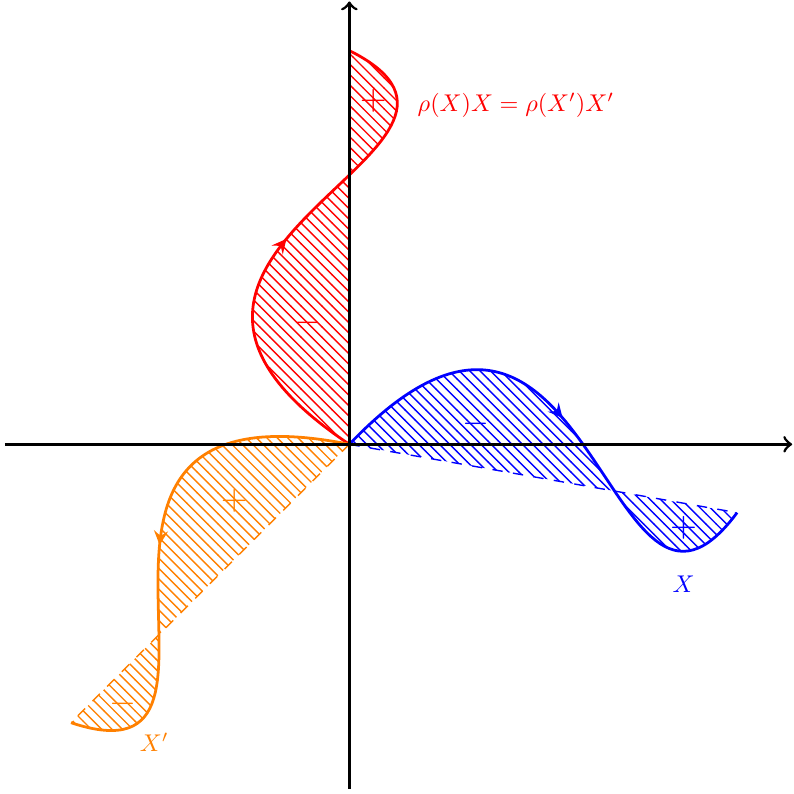}
    \caption{Applying the moving-frame map for planar curves to two paths \mbox{$X$ and $X'$} lying on the same $\mathrm{O}_2(\R)$ orbit}
    \label{fig:area}
\end{figure}

Any path $\CURVE$ in $\R^2$ defines an element $\ccn(\CURVE) = \proj_{\le n}(\log(\IIS(\CURVE)))\in \lie{n}2$. Since $\rhotwon(\ccn) = \rhotwo(\cctwo)$, we can define the invariantization of $\CURVE$ with respect to $\O_2(\R)$ as $Y := \rhotwo(\cctwo(\CURVE))\cdot \CURVE$. The coordinates of $\log(\IIS(Y))$ as functions of the coordinates of $\log(\IIS(\CURVE))$ are invariant functions for paths under $\O_2(\R)$.

A (Lyndon) basis for $\lie{4}2$ corresponds to the coordinates (see \Cref{ex:lyndon})
$$
\cc_4 = (c_{1}, c_{2}, c_{{12}}, c_{112},c_{122},c_{1112},c_{1122},c_{1222}).
$$
As detailed before in Section \ref{sec:MF2}, we have that
$$
c_{1}(Y) = 0, \quad c_{2}(Y) = \sqrt{c_{1}(\CURVE)^2+c_{2}(\CURVE)^2}, \quad c_{12}(Y) = |c_{12}(\CURVE)|.
$$
Using the action as defined in Section \ref{ss:invariants}, one can compute
\begin{align*}
c_{112}(Y) &=\frac{c_{1}(\CURVE)c_{122}(\CURVE)+c_{112}(\CURVE)c_{2}(\CURVE)}{\sqrt{c_{1}(\CURVE)^2+c_{2}(\CURVE)^2}}  \\
 c_{122} (Y)&= \text{sgn}(c_{12})\left(\frac{-c_{1}(\CURVE)c_{112}(\CURVE)+c_{122}(\CURVE)c_{2}(\CURVE)}{\sqrt{c_{1}(\CURVE)^2+c_{2}(\CURVE)^2}}\right) \\
c_{1112}(Y) &= \text{sgn}(c_{12})\left(\frac{c_{1}(\CURVE)^2c_{1222}(\CURVE)+c_{1}(\CURVE)c_{2}(\CURVE)c_{1122}(\CURVE)+c_{2}(\CURVE)^2c_{1112}(\CURVE)}{c_{1}(\CURVE)^2+c_{2}(\CURVE)^2}\right)\\
 c_{1122}(Y)&=\frac{-c_{1}(\CURVE)^2c_{1122}(\CURVE)+2c_{1}(\CURVE)c_{2}(\CURVE)(c_{1222}(\CURVE)-c_{1112}(\CURVE))+c_{2}(\CURVE)^2c_{1122}(\CURVE)}{c_{1}(\CURVE)^2+c_{2}(\CURVE)^2}\\
  c_{1222}(Y)&= \text{sgn}(c_{12})\left(\frac{c_{1}(\CURVE)^2c_{1112}(\CURVE)-c_{1}(\CURVE)c_{2}(\CURVE)c_{1122}(\CURVE)+c_{2}(\CURVE)^2c_{1222}(\CURVE)}{c_{1}(\CURVE)^2+c_{2}(\CURVE)^2}\right).
\end{align*}

As before, for any two paths $\CURVE$ and $\tilde{\CURVE}$ starting at the origin, we have that $\cc_4(\CURVE)$ is related to $\cc_4(\tilde{\CURVE})$ under $\O_2(\R)$ if and only if $\cc_4(Y) =\cc_4(\tilde{Y})$. By inspection, we see that a simpler set of \textit{polynomial} invariants also determine the equivalence class of the image of a path $\CURVE$ in $\lie{4}2$.

\begin{align*}
p_1(\CURVE) &= c_{1}(\CURVE)^2+ c_{2}(\CURVE)^2\\
p_2(\CURVE) &= c_{12}(\CURVE)^2\\
p_3(\CURVE) &= c_{1}(\CURVE)c_{122}(\CURVE)+c_{112}(\CURVE)c_{2}(\CURVE)\\
p_4(\CURVE) & = c_{12}(\CURVE)\left(-c_{1}(\CURVE)c_{112}(\CURVE)+c_{122}(\CURVE)c_{2}(\CURVE)\right)\\
p_5(\CURVE) & =  c_{12}(\CURVE)\left(c_{1}(\CURVE)^2c_{1222}(\CURVE)+c_{1}(\CURVE)c_{2}(\CURVE)c_{1122}(\CURVE)+c_{2}(\CURVE)^2c_{1112}(\CURVE)\right)\\
p_6(\CURVE) &= -c_{1}(\CURVE)^2c_{1122}(\CURVE)+2c_{1}(\CURVE)c_{2}(\CURVE)(c_{1222}(\CURVE)-c_{1112}(\CURVE))+c_{2}(\CURVE)^2c_{1122}\\
p_{7}(\CURVE) & = c_{12}(\CURVE)\left(c_{1}(\CURVE)^2c_{1112}(\CURVE)-c_{1}(\CURVE)c_{2}(\CURVE)c_{1122}(\CURVE)+c_{2}(\CURVE)^2c_{1222}(\CURVE)\right)
\end{align*}

The value of $\CURVE$ on the above invariant set determines the value of $\cc_4(Y)$. Thus they provide a simpler invariant representation for $\cc_4(\CURVE) = \proj_{\le 4}(\log(\IIS(\CURVE)))$.

\begin{remark}
 It is an interesting fact that by adding the invariants $c_{1112}(Y)$ and $c_{1222}(Y)$, we get the much simpler invariant
 \begin{equation*}
  c_{1112}(Y)+c_{1222}(Y)=\text{sgn}(c_{12})(c_{1112}(\CURVE)+c_{1222}(\CURVE)).
 \end{equation*}
 In the polynomial invariant set, one can likewise replace either $p_4$ or $p_{7}$ by
 \begin{equation*}
  p_4'(\CURVE)=c_{12}(\CURVE)(c_{1112}(\CURVE)+c_{1222}(\CURVE)).
 \end{equation*}
\end{remark}

We already know from Example \ref{ex:almost-polynomial} in combination with Proposition \ref{prop:almost-polynomial} that $\rhotwon$ is an almost-polynomial moving frame and thus leads to a fundamental set of polynomial invariants.
However, we prove here the following stronger statement with a slightly different construction on the level of the individual coordinates of the first kind $c_{h}$. 

\begin{theorem}\label{thm:polinvariants2d}
 There exists a set of polynomials $\pol_h$ and a polynomial map $\polmap:\,\lie{n}2\to\lie{n}2$ which is bijective when restricted to $\mathcal{U}_{2;\le n}$ such that
 \begin{equation*}
  q(\ccn(\rhotwon(Z)\cdot Z))=(\pol_h(\ccn(\CURVE)))_h.
 \end{equation*}
 Thus, the $\pol_h(\ccn)$ form a set of polynomial invariants determining the equivalence class of a path $\CURVE$ in $\lie{n}2$.
\end{theorem}
This theorem is stronger in the sense that the map $q$ relating the two invariant sets is also shown to be polynomial, constrasted with Theorem $\ref{prop:almost-polynomial}$ where did not assume the form of the map
\begin{equation*}
\varrho(\ccn)\cdot\ccn\mapsto\lambda(\ccn)\varrho(\ccn)\cdot\ccn.
\end{equation*}
\begin{proof}
 Let $n(\w{i},w)$ denote the number of times the letter $\w{i}$ appears in the word $w$.
 Since $c_h(\CURVE)=\langle\IIS(\CURVE),\zeta_h\rangle$ for unique $\zeta_h\in T(\mathbb{R}^d)$, where each $\zeta_h$ is a linear combination of permutations of the word $h$, and since $B$ is diagonal, we have
 \begin{equation*}
  c_h(\rhotwon(\ccn(\CURVE))\cdot\CURVE)=(\mu_2)_{11}^{n(\w{1},h)}(\mu_2)_{22}^{n(\w{2},h)}c_h(\nu_2(\ccn(\CURVE)\CURVE).
 \end{equation*}
 Let $m(w)=0$ if $n(\w{1},w)$ is even and $m(w)=1$ if $n(\w{1},w)$ is odd. Then, on $U^2_n$,
 \begin{equation*}
  c_{12}(\rhotwon(\ccn(\CURVE))\cdot\CURVE)^{m(h)}c_{2}(\rhotwon(\ccn(\CURVE))\cdot\CURVE)^{n(\w{1},h)}(\mu_2)_{11}^{n(\w{1},h)}=c_{12}(\CURVE)^{m(h)}
 \end{equation*}
 and 
 \begin{equation*}
  c_{2}(\rhotwon(\ccn(\CURVE))\CURVE)^{n(\w{2},h)}(\mu_2)_{22}^{n(\w{2},h)}=1.
 \end{equation*}
 Thus, since $c_h(C(\CURVE)\CURVE)$ is polynomial in $\cc_n(\CURVE)$, also
 \begin{align*}
  \pol_h(\ccn(\CURVE))&:= c_{12}(\rhotwon(\ccn(\CURVE))\cdot\CURVE)^{m(h)}c_{2}(\rhotwon(\ccn(\CURVE))\cdot\CURVE)^{|h|}c_h(\rhotwon(\ccn(\CURVE))\cdot\CURVE)\\
  &\,\,=c_{12}(\CURVE)^{m(h)}c_h(\nu_2(\ccn(\CURVE))\cdot\CURVE)
 \end{align*}
 is polynomial in $\cc_n(\CURVE)$, and also polynomial in $(c_h(\rhotwon(\ccn(\CURVE))\cdot\CURVE))_h$. Finally, all $c_h(\rhotwon(\ccn(\CURVE))\cdot\CURVE)$ can be retrieved from $(\pol_h(\ccn(\CURVE)))_h$ via
 \begin{equation*}
  c_h(\rhotwon(\CURVE)\cdot\CURVE)=\frac{\pol_h(\ccn(\CURVE))\sqrt{\pol_{{2}}(\ccn(\CURVE))}^{m(h)}}{\sqrt{\pol_{{12}}(\ccn(\CURVE))}^{m(h)}\sqrt{\pol_{{2}}(\ccn(\CURVE)}^{|h|}}.
 \end{equation*}

\end{proof}
We present an example here of how the resulting invariants differ when obtained using the procedures of Proposition \ref{prop:almost-polynomial} and of Theorem \ref{thm:polinvariants2d}. For the non-zero coordinates of $\lambda(\cc_{\le 3})\rhotwon(\cc_{\le 3}\cdot\cc_{\le 3})$, we get
\begin{align*}
    \hat c_{2}&=c_{1}^2+c_{2}^2, \quad\hat c_{12}=c_{12}^2(c_{1}^2+c_{2}^2), \quad\hat c_{112}=c_{12}^2(c_{1}^2+c_{2}^2)(c_{1}c_{122}+c_{112}c_{2})\\
    \hat c_{122}&=c_{12}(c_{1}^2+c_{2}^2)(-c_{1}c_{112}+c_{122}c_{2}),
\end{align*}
while
\begin{align*}
    \pol_{2}(\ccn)&=c_1^2+c_2^2,\quad \pol_{12}(\ccn)=c_{12}^2(c_{1}^2+c_{2}^2),\quad \pol_{112}(\ccn)=(c_{1}^2+c_{2}^2)(c_{1}c_{122}+c_{112}c_{2}),\\
    \pol_{122}&=c_{12}(c_{1}^2+c_{2}^2)(-c_{1}c_{112}+c_{122}c_{2}),
\end{align*}
Thus up to level $4$, they only differ in the $112$ coordinate, which is a bit `simpler' than the resulting invariant using the procedure of Theorem \ref{thm:polinvariants2d}. Looking at the previous $p_i$s we listed, we see how this polynomial set can be further simplified. However, we lack a general algorithm for a `full' simplification of the set of polynomial invariants. This would be achieved if they form a minimal algebra generating set for $\R[\lie{n}2]^{\O_2(\R)}$. This is an interesting investigation for future research.

\subsection{Spatial curves}\label{sec:inv.spatial}
We can define a moving frame similarly for $\lie{n}3$. Theorem \ref{thm:MainThm} shows that the subset of $\lie{n}3$ defined by
$$
\K_{3;\le n} = \left\{c_{1}=c_{2}=c_{13}=0,  c_{3}, c_{12}, c_{23} >0\right\}
$$
is a cross-section for the action of $\O_3(\R)$ on a Zariski-open subset of $\lie{n}3$. Define the following polynomials
in coordinates of $\lie{2}3$,
\begin{align*}
p_1(\ccn) &= c_{1}^2+c_{2}^2+c_{3}^2\\
p_2(\ccn)&= c_{1}^2(c_{12}^2+c_{13}^2) + 2c_{1}c_{23}(c_{13}c_{2}-c_{12}c_{3})+c_{2}^2(c_{12}^2+c_{23}^2)+2c_{12}c_{13}c_{2}c_{3}+c_{3}^2(c_{13}^2+c_{23}^2)\\
p_3(\ccn) &= c_{1}c_{23}-c_{2}c_{13}+c_{3}c_{12}.
\end{align*}

Note that $p_2(\ccn)\geq 0$ for all $\ccn\in\lie{n}3$. There are two ways to see this: First by Cauchy-Schwartz via
\begin{equation*}
 p_2(\ccn)=p_1(\ccn)(c_{12}^2+c_{13}^2+c_{23}^2)-p_3(\ccn)^2=\|v_1\|_2^2\|v_2\|_2^2-(v_1\cdot v_2)^2,
\end{equation*}
where $v_1=(c_{1},c_{2},c_{3})^{\top}$, $v_2=(c_{23},-c_{13},c_{12})^{\top}$. On the other hand, it can also be written as a sum of squares,
\begin{equation*}
 p_2(\ccn)=(c_{12}c_{1}-c_{23}c_{3})^2+(c_{13}c_{1}+c_{23}c_{2})^2+(c_{12}c_{2}+c_{13}c_{3})^2,
\end{equation*}
revealing that it is nothing but $Mv\cdot Mv$ from Example~\ref{ex:generatorsO3field}, while $p_1(\ccn)$ is $v\cdot v$.

Using $p_1(\ccn),p_2(\ccn)\geq 0$, one can check that the group element $\rhothreen(\ccn)\in \O_3(\R)$ defined by
\begin{equation}\label{eq:movingframed3}
\rhothreen(\ccn) = \mu_3(\ccn) \nu_3(\ccn),
\end{equation}
where
$$
\mu_3(\ccn) = \begin{bmatrix}
\frac{\text{sgn}(p_3(\ccn))}{\sqrt{p_1(\ccn)p_2(\ccn)}}&0 &0\\[2ex]
0& \frac{1}{\sqrt{p_2(\ccn)}} &0\\[2ex]
0&0 &\frac{1}{\sqrt{p_1(\ccn)}}
\end{bmatrix}
$$
and
$$
\nu_3(\ccn) =
\scriptsize\begin{bmatrix}
c_{1}(-c_{13}c_{2}+c_{12}c_{3})-c_{23}(c_{2}^2+c_{3}^2) &c_{1}^2c_{13}+c_{1}c_{2}c_{23}+c_{3}(c_{12}c_{2}+c_{13}c_{3}) & -c_{1}^2c_{12}-c_{1}c_{23}c_{3}+c_{2}(c_{12}c_{2}+c_{13}c_{3})\\[1ex]
c_{12}c_{2}+c_{13}c_{3} & -c_{1}c_{12}+c_{23}c_{3} &-c_{1}c_{13}-c_{2}c_{23}\\[1ex]
c_{1} & c_{2} & c_{3}
\end{bmatrix},
$$
brings any element $\ccn\in\lie{n}3$ such that $p_2(\ccn),p_3(\ccn)\neq 0$ to $\mathcal{K}$. In particular the moving frame map $\rhothreen:U_{3; \le n}\rightarrow\O_3(\R)$ is defined on the Zariski-open set $U_{3; \le n} = \left\{p_2(\ccn), p_3(\ccn)\neq 0\right\}\subset \lie{n}3$.
Note that $p_3(\ccn)\neq0$ already implies $p_1(\ccn)\neq0$. Furthermore note that due to the fact that $\rhothreen(\ccn)\in\O_3(\R)$, we have that the sum of the squares of the first row of $\nu_3(\ccn)$ is $p_1(\ccn)p_2(\ccn)$, the sum of the squares of the second row is $p_2(\ccn)$ and the sum of the squares of the third row is $p_1(\ccn)$.

Then for any path $\CURVE\in \R^3$ and $Y :=\rho(\cctwo(\CURVE)) \cdot \CURVE$, the non-zero coordinates of $\cctwo(Y)$ are invariant functions given by%
\footnote{
We note that $p_3(\ccn(\CURVE))$
is the ``signed volume'' of the curve,
compare \cite[Lemma 3.17]{DR2019}.}
\begin{align*}
c_{3}(Y)&=\sqrt{p_1(\cctwo(\CURVE))},\\
c_{12}(Y)&= \frac{|p_3(\cctwo(\CURVE))|}{\sqrt{p_1(\cctwo(\CURVE))},},\\
c_{23}(Y) &= \sqrt{\frac{p_2(\cctwo(\CURVE))}{p_1(\cctwo(\CURVE))}}.
\end{align*}

From this we can conclude that the polynomial invariants $p_1(\ccn(\CURVE)), p_2(\ccn(\CURVE)),$ and $p_3(\ccn(\CURVE))^2$ characterize the equivalence class of $\cctwo(\CURVE)$ under $\O_3(\R)$.

Furthermore, we can, in this special case, determine the \textit{global} form of the $f_i$ from Proposition \ref{prop:ComplexCS}:
\begin{align*}
    f_1(\cctwo)&=\hat{c}_{3}^2=p_1(\cctwo),\\
    f_2(\cctwo)&=\hat{c}_{{23}}^2=\frac{p_2(\cctwo)}{p_1(\cctwo)},\\
    f_3(\cctwo)&=\hat{c}_{{12}}^2=\frac{p_3(\cctwo)^2}{p_1(\cctwo)},
\end{align*}
where $\hat{c}_{3}$, $\hat{c}_{12}$, $\hat{c}_{{23}}$ denote the non-zero components of $\rhothree(\cctwo)\cdot\cctwo$. To show almost-polynomiality, let $\lambda(\ccn)$ be the diagonal matrix with entries
\begin{equation*}
    |p_3(\ccn)|\sqrt{p_1(\ccn)p_2(\ccn)},\quad \sqrt{p_2(\ccn)},\quad\sqrt{p_1(\ccn)}.
\end{equation*}
The function $\lambda_3$ is $\O_3(\R)$-invariant since $p_1$, $p_2$ and $p_3^2$ are invariant. Furthermore $\lambda_3(\ccn)\mu_3(\ccn)$ is diagonal with entries
\begin{equation*}
 p_3(\ccn), \quad 1,\quad 1,
\end{equation*}
and hence  $\lambda_3(\ccn)\rhothreen(\ccn)=\lambda_3(\ccn)\mu_3(\ccn)\nu_3(\ccn)$ is also polynomial in $\ccn$. The non-zero coordinates of $\lambda_3(\ccn)\rho_3(\ccn)\cdot\cctwo$ are then given by
\begin{equation*}
    \hat c_{3}=p_1(\ccn),\quad \hat c_{12}=p_2(\ccn)p_3(\ccn)^2,\quad \hat c_{23}=p_2(\ccn).
\end{equation*}
Thus we have $\lambda_3(\ccn)=\kappa_3(\lambda_3(\ccn)\rho_3(\ccn)\cdot\ccn)$ with $\kappa_3(\hat{\mathbf{c}}_{\leq n})$ as the diagonal matrix with entries
\begin{equation*}
    \sqrt{\hat{c}_{3}\hat{c}_{12}},\quad \sqrt{\hat{c}_{23}},\quad\sqrt{\hat{c}_3}.
\end{equation*}
Hence, $\rhothreen$ is an almost-polynomial moving frame, leading to a fundamental set of polynomial invariants.

\begin{example}
  \label{ex:mfmomentcurve}
  
  Continuing with our running example, the moment curve, we have already seen (\Cref{ex:momentCurve}) that 
	\[\proj_{\le 2}\log\IIS(\CURVE)=
		\left(\begin{bmatrix}1\\1\\1\end{bmatrix},\begin{bmatrix}0&\frac16&\frac14\\-\frac16&0&\frac1{10}\\-\frac14&-\frac1{10}&0\end{bmatrix}\right)
	\]
	The matrix
	\[
		A_1=\begin{bmatrix}
			\frac1{\sqrt{6}}&\frac{1}{\sqrt{6}}&-\sqrt{\frac{2}{3}}\\
			-\frac{1}{\sqrt{2}}&\frac{1}{\sqrt{2}}&0\\
			\frac{1}{\sqrt{3}}&\frac{1}{\sqrt{3}}&\frac{1}{\sqrt{3}}
		\end{bmatrix}
	\]
	is such that
	\[A_1\cdot\proj_{\le 2}\log\IIS(\CURVE)=\left(\begin{bmatrix}0\\0\\\sqrt{3}\end{bmatrix},
			\begin{bmatrix}
				0&\frac{1}{60\sqrt{3}}&\frac{7}{20\sqrt{2}}\\
				-\frac{1}{60\sqrt{3}}&0&-\frac{29}{60\sqrt{6}}\\
				-\frac{7}{20\sqrt{2}}&\frac{29}{60\sqrt{6}}&0
			\end{bmatrix}
		\right).
	\]
	Note that this is an element of $L^{(1)}_3(\R)\subset\LC{3}1$,
	\eqref{eq:Ls}.

	Finally, the matrix
	\[
		A_2=\begin{bmatrix}
			-\frac{29}{2\sqrt{541}}&-\frac{21\sqrt{3}}{2\sqrt{541}}&0\\
			\frac{21\sqrt{3}}{2\sqrt{541}}&-\frac{29}{2\sqrt{541}}&0\\
			0&0&1
		\end{bmatrix}
	\]
	is such that
	\[ A_2\cdot(A_1\cdot \proj_{\le 2}\log\IIS(\CURVE))=\left(\begin{bmatrix}0\\0\\\sqrt{3}\end{bmatrix},
			\begin{bmatrix}
				0&\frac{1}{60\sqrt{3}}&0\\
				-\frac{1}{60\sqrt{3}}&0&\frac{\sqrt{541}}{30\sqrt{6}}\\
				0&-\frac{\sqrt{541}}{60\sqrt{6}}&0
			\end{bmatrix}
		\right)\in\K_{3;\le 2}\subsetneq\LC{3}2(\R).
	\]

Figure \ref{fig:moment_curve} shows the effects of these transformations on the path itself.
	\begin{figure}[!ht]
		\centering
		\includegraphics[width=0.5\textwidth]{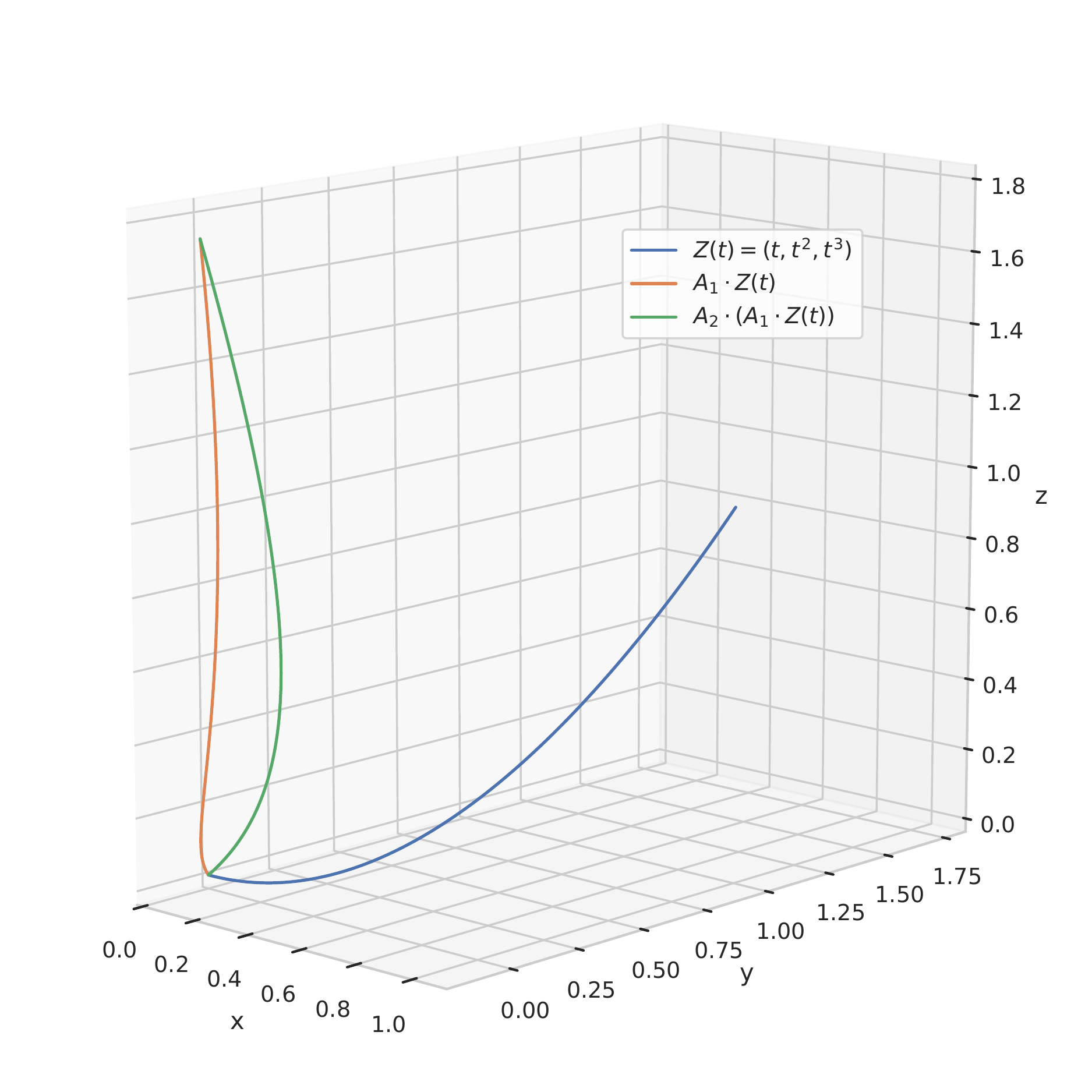}\includegraphics[width=0.5\textwidth]{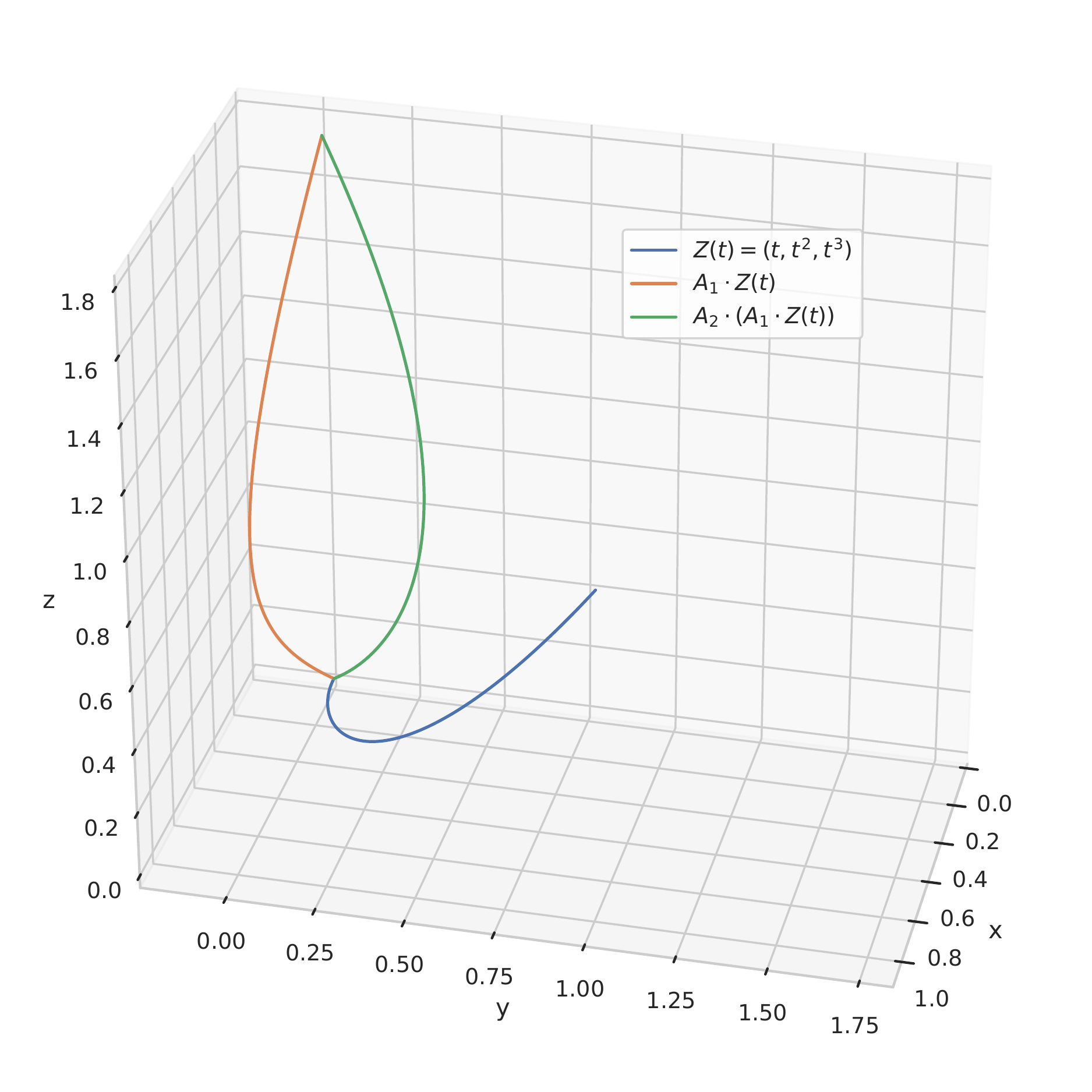}\caption{}
		\label{fig:moment_curve}
	\end{figure}
	
	In this sense, this two step process is similar to the iterative process outlined in the proof of Proposition \ref{prop:ComplexCS} of bringing an element of $\lie{2}d$ to successively smaller linear spaces. The transformation $A_1$ brings $\cctwo(\CURVE)$ to $\LC{3}1(\mathbb{R})\subset\LC{3}1$, then finally to $\K_{3;\le 2}\subsetneq\LC{3}2(\mathbb{R})\subset\LC{3}2$ by a transformation $A_2$. In principle, given a procedure to rotate an element of $\R^d$ to a particular axis, this iterative process is quite easy to perform to bring any $\cctwo(\CURVE)$ for any path $\CURVE$ to $\CROSS$, and hence invariantize any path.

	Alternatively one can directly use the moving frame map in \eqref{eq:movingframed3}; note that this is equivalent to the single action by the matrix
	\[
		\rhothree(\cctwo(\CURVE))=A_2A_1= \begin{bmatrix}
 \frac{17}{\sqrt{3246}} & -23 \sqrt{\frac{2}{1623}} & \frac{29}{\sqrt{3246}} \\
 \frac{25}{\sqrt{1082}} & -2 \sqrt{\frac{2}{541}} & -\frac{21}{\sqrt{1082}} \\
 \frac{1}{\sqrt{3}} & \frac{1}{\sqrt{3}} & \frac{1}{\sqrt{3}}\end{bmatrix}.
\]

	Figure \ref{fig:moment_proj} shows the projection of \(\hat{Z}:=\rhothree(\mathbf c_2(\CURVE))\cdot \CURVE\) onto the \( (x,z)\) plane.
	One can check that the total area under the curve vanishes.
	
	\begin{figure}[!ht]
		\centering
		\vspace{1ex}
		\includegraphics[width=0.87\textwidth]{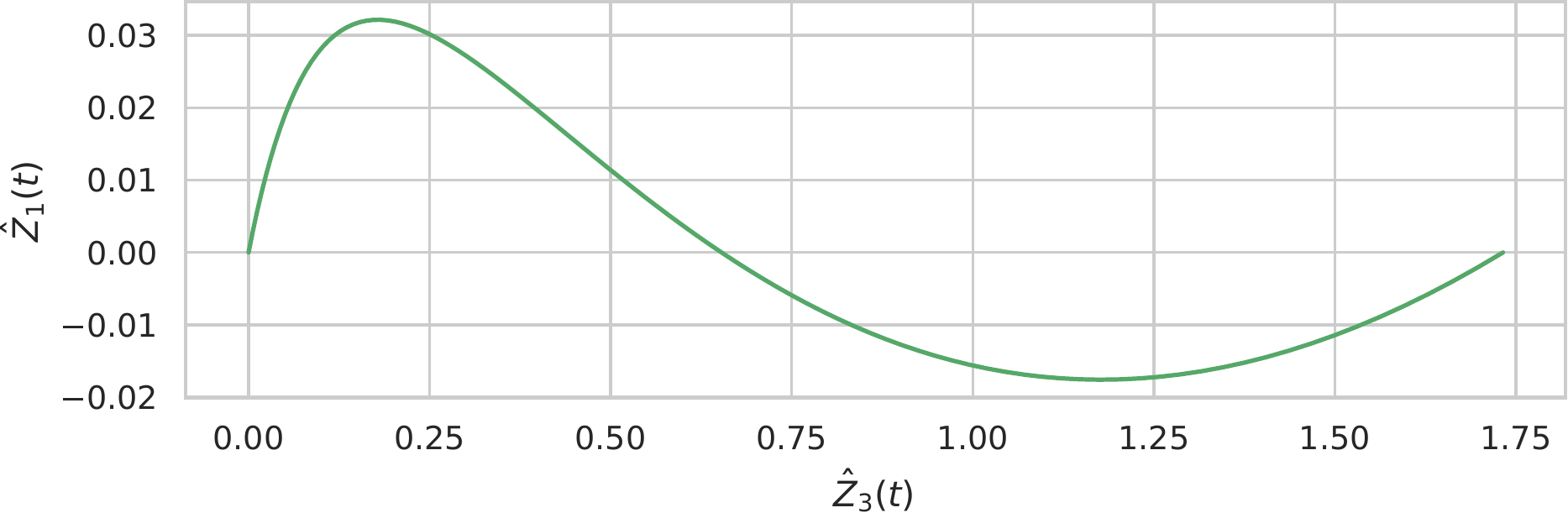}\hspace{10ex}
		\caption{}
		\label{fig:moment_proj}
	\end{figure}
\end{example}

\FloatBarrier

\section{Discussion and open problems}\label{sec:discussion}
We conclude with a discussion of some interesting questions arising from this work. We presented a method to construct $\O_d(\R)$ invariants for a path $\CURVE$ from the coordinates of the log signature (of the iterated-integrals signature) in a way that completely characterizes the orbit of $\proj_n(\log(\IIS(\CURVE))$ (or $\proj_n(\IIS(\CURVE))$) under $\O_d(\R)$. This procedure also furnishes a quick method to compare equivalence classes of paths under $\O_d(\R)$ without computing the full set of invariants (see Example \ref{ex:mfmomentcurve}).

In particular Theorem \ref{thm:IISequiv} is similar in spirit to \cite[Conjecture 7.2]{DR2019}, where the authors characterize all \textit{linear} $\SOdR$-invariants in the coordinates of $\IIS(\CURVE)$ and ask if these determine a path up to $\SOdR$ and tree-like extensions. The invariant sets we construct are smooth functions in the coordinates of $\log(\IIS(\CURVE))$, though in many cases we can, by inspection, find an equivalently generating polynomial set (see Section \ref{sec:inv.planar}). Polynomials in coordinates of $\log(\IIS(\CURVE))$ correspond to polynomial invariants in the coordinates of $\IIS(\CURVE)$, which yield linear $\OdR$-invariants by the shuffle relations. Thus the conjecture remains open, and more broadly the connection between the two sets of invariants should be explored.

In \Cref{sec:complexdetour}, we investigate sets of separating sets of rational and polynomial invariants for the action of $\OdR$ on $\lie{2}d$. An open question is whether the polynomial invariants we construct, generate the \textit{ring} of polynomial invariants for this action. In even more generality questions remain about the relationship between the polynomial invariants we construct and the ring of polynomial invariants for the action of $\OdR$ on $\lie{n}d$.

Additionally we only consider $\OdR$-invariants (and to a lesser extend $\SO_d(\R)$) in this work. The dimension of $\OdR$ implies that to construct a cross-section for the action, one only has to consider the action on $\lie{2}d$. For larger groups like $\GL_d(\R)$ one may have to construct a cross-section using coordinates on $\lie{3}d$.

The cross section $\mathcal{K}$ in Section \ref{sec:MF} can also be used as a starting point for groups containing $\OdR$, since any element of $\lie{2}d$ can be brought to $\mathcal{K}$ by an element of $\OdR$. For instance if one considers scaling transformations in addition to orthogonal transformations, changing the conditions of $c_{d}, c_{i(i+1)} >0$ on $\mathcal{K}$ to $c_{d}=c_{i(i+1)}=0$, for $1\leq i < d$, likely yields a cross-section.

In Section \ref{sec:towardpoly} we introduce Conjecture \ref{conj:almostpoly}, which we prove holds for $d=2$ in Theorem \ref{thm:polinvariants2d} and for $d=3$ in Section \ref{sec:inv.spatial}. In general the invariants produced by the moving frame method are only guaranteed to be smooth, and hence proving this conjecture for $d>3$ is of interest. In particular, polynomial invariant functions of iterated-integral values are desired because they can be expressed as elements of $T(\R^d)$, they provide the simplest and most structured (graded) way of looking at invariants, with an immediate connection to polynomial algebraic geometry, and it is widely assumed and partially proven that they are sufficient for characterization of orbits, see the discussion of the Diehl-Reizenstein conjecture in Section \ref{sec:towardpoly}. %

As mentioned in the introduction, there are many applications of the iterated-integrals signature of paths where finding $\OdR$-invariant features could be advantageous. It would be interesting to see if the sets of integral invariants constructed, or ``invariantization'' procedure outlined can be useful for such applications.

\bibliographystyle{alpha}
\bibliography{notes}	

\end{document}